\documentclass[a4paper, 12pt]{amsart}
\usepackage{amsmath}
\usepackage{amssymb}
\usepackage{amsthm, enumerate}

\usepackage[top=20truemm, left=17truemm, right=17truemm]{geometry}
\usepackage{amscd}
\usepackage{ulem}
\usepackage{xcolor}
\usepackage[arrow,matrix,curve,cmtip,ps]{xy}
\usepackage{comment}

\usepackage{amscd}
\newtheorem{dfn}{Definition}[section]
\newtheorem{thm}[dfn]{Theorem}
\newtheorem{lem}[dfn]{Lemma}
\newtheorem{remark}[dfn]{Remark}
\newtheorem{cor}[dfn]{Corollary}

\newtheorem{prop}[dfn]{Proposition}
\newtheorem{ex}[dfn]{Example}
\newcommand{\Rist}[2]{\operatorname{Rist}_{#1}{(#2)}}

\title[Topological full groups arising from  Cuntz and Cuntz--Toeplitz algebras]{Topological full groups arising from Cuntz and Cuntz--Toeplitz algebras and their crossed products}
\author{Ryoya Arimoto}
\address{RIMS, Kyoto University, Sakyo-ku, Kyoto, 606-8502 Kyoto, Japan}
\email{arimoto@kurims.kyoto-u.ac.jp}

\author{Taro Sogabe}
\address{Graduate School of Science, Kyoto University, Sakyo-ku, Kyoto, 606-8502, Japan}
\email{sogabe.taro.3v@kyoto-u.ac.jp}

\date{}
\begin{document}

\begin{abstract}
    %In the operator algebra theory,
    %The Higman--Thompson groups are realized as the topological full groups arising from the graph groupoids modeling the Cuntz algebras.
    In this paper,
    we investigate the topological full groups arising from the Cuntz and Cuntz--Toeplitz algebras and their crossed products with the Cartan subalgebras of Cuntz and Cuntz--Toeplitz algebras.
    We study the normal subgroups and abelianization of these groups and completely determine the KMS states of the crossed products with respect to the canonical gauge actions.
\end{abstract}

\maketitle

%%%%%%%%%%%%%%%%%%%%%%%%%%%%%%%%%%%%%%%%%%%%%%%%%%%%%%%%%%%%%%%%%%%%%%%%%%%%%%%%%%%%%%%%%%%%%%%%%%%%%%%%%%%%%%%%%%%%%%%%%%%%%%%%%%%%%%%%%%%%%%%%%%%%%%
\section{Introduction}
The Cuntz algebra $\mathcal{O}_n$ ($n\geq 2$), defined by J. Cuntz \cite{Cu}, has played an important role in the theory of operator algebras, especially in the classification theory of $\mathrm{C}^*$-algebras.
This algebra $\mathcal{O}_n$ is the universal $\mathrm{C}^*$-algebra generated by mutually orthogonal isometries $\{ S_1, \ldots, S_n \}$ that satisfy the Cuntz relation $\sum _{i=1}^n S_i S_i^* =1$, and is known to be separable, simple, nuclear, and purely infinite.
%These algebras are universal $\mathrm{C}^*$-algebras generated by certain isometries and are known to be separable, simple, nuclear, and purely infinite.
They have been studied from various perspectives, mainly from classification theory and the connections between groupoids.
Especially,
we are interested in the groups, called topological full groups, arising from the groupoid pictures for the Cuntz algebras.
%In this paper, we will investigate an analogue of this algebra from the latter perspective.

Richard J. Thompson introduced the groups $F$, $T$, and $V$ in his unpublished note in 1965, motivated by constructing finitely presented groups with unsolvable word problems.
In the same note, he showed that these groups are all finitely presented, and $T$ and $V$ are simple, and this gave us the first examples of finitely presented infinite simple groups.
Since these groups have such interesting properties, they have attracted considerable attention, and their generalizations have also been studied.
The Higman--Thompson groups $V_n$, which were introduced by G. Higman in 1974 \cite{Hig}, are one of the generalizations.
These groups are known to appear in many different contexts, and in this paper, we will focus on the relationship with operator algebras, especially with groupoids.

%The Thompson's group $V$, which is introduced by Richard J. Thompson in his unpublished note

The study of KMS states on a $\mathrm{C}^*$-algebra with a time evolution has been conducted by many researchers to date.
Though the notion of KMS states originated from physics, they have been studied from a mathematical motivation.
Let $A$ be a $\mathrm{C}^*$-algebra and $\gamma \colon \mathbb{R} \curvearrowright A$ be a $\mathbb{R}$-action on it.
We say that an $\gamma$-invariant state on $A$ is a $\mathrm{KMS}_\beta$-state if it satisfies a generalized tracial conditions (see Sec. \ref{gac}). %\textcolor{red}{def wo kaku}
In some cases, the structure of KMS states is completely determined.
For instance, D. Olesen and G. K. Pedersen showed that $\mathrm{KMS}_\beta$-state with respect to the canoncial gauge action on the Cuntz algebra $\mathcal{O}_n$ exists if and only if $\beta = \log n$ and the $\mathrm{KMS}_{\log n}$-state is unique.
%\textcolor{red}{jinmei no hyouki wo touitu}
%More generally, KMS states with respect to other $\mathbb{R}$-actions on the Cuntz algebra and KMS states on some generalizations of the Cuntz algebra have been computed (see, for instance, \textcolor{red}{citation}).
The ground states, which can be interpreted as $\mathrm{KMS}_{\infty}$-states, have also been studied and are completely determined in some cases.
There are many other important previous results, including the Cuntz--Kireger algebras and the crossed products,
and the reader may refer to \cite{Nes, NY, VC, Evans, FEW, LRS} and references therein.

The first two notions, the Cuntz algebras and the Higman--Thompson groups, can be understood through groupoids.
Groupoids are regarded as a generalization of groups and topological dynamics.
From a (\'etale) groupoid $\mathcal{G}$, one can construct a reduced groupoid $\mathrm{C}^*$-algebra $\mathrm{C}^*_r (\mathcal{G})$ and a topological full group $[[ \mathcal{G}]]$.

Here, the groupoid C*-algebra is some kind of group ring, and these topological full groups are the groups of symmetries of the dynamics (see Sec. \ref{s2}).
There are so-called graph groupoids obtained from the directed graphs,
and we have a graph $O_n$ whose graph groupoid $\mathcal{G}_{O_n}$ realizes $\mathcal{O}_n\cong C^*_r(\mathcal{G}_{O_n})$ and $V_n\cong [[\mathcal{G}_{O_n}]]$.

%A topological full group, introduced by Matui (\textcolor{red}{citation}), realizes the dynamics of a groupoid as a genuine topological dynamics.
By definition,
the topological full group naturally acts on the unit space $\mathcal{G}^{(0)}$ of the groupoid,
and the transformation groupoid $\mathcal{G}^{(0)} \rtimes [[\mathcal{G}]]$ %of the action of a topological full group $[[ \mathcal{G} ]]$ on a unit space $\mathcal{G}^{(0)}$ 
and $\mathcal{G}$ are related via the canonical groupoid homomorphism $\mathcal{G}^{(0)} \rtimes [[\mathcal{G}]] \to \mathcal{G}$.
This leads us to believe that the reduced crossed product $C(\mathcal{G}^{(0)}) \rtimes _r [[ \mathcal{G} ]]$ %which is the reduced groupoid $\mathrm{C}^*$-algebra of the transformation groupoid $\mathcal{G}^{(0)} \rtimes [[\mathcal{G}]]$ 
and the reduced groupoid $\mathrm{C}^*$-algebra $\mathrm{C}^*_r(\mathcal{G})$ might share certain properties, which is the original motivation of this paper.
By using the full groupoid C*-algebras,
these algebras are related as follows, where the full and reduced crossed products coincide if the transformation groupoids are amenable, which is not the case in this paper:
\[\xymatrix{
C(\mathcal{G}^{(0)})\rtimes [[\mathcal{G}]]\ar[r]\ar[d]&C^*(\mathcal{G})\ar[d]\\
C(\mathcal{G}^{(0)})\rtimes_r[[\mathcal{G}]]&C^*_r(\mathcal{G})
}\]
%Since there is a groupoid $\mathcal{G}_{O_n}$ such that $\mathrm{C}^*_r (\mathcal{G}_{O_n}) \cong \mathcal{O}_n$ and $[[ \mathcal{G}_{O_n}]] \cong V_n$, they 

%%%%%%%%%%%%%%%%%%%%%%%%%%%%%%%%%%%%%%%%%%%%%%%%%%%%%%%%%%%%%%%%%%%%%%%%%%%%%%%%%%%%%%%%%%%%%%%%%%%%%%%%%%%%%%%%%%%%%%%%%%%%%%%

%En GEn [[GEn]]

\subsection*{Our results}

In this paper,
we will investigate Cuntz--Toeplitz analogue of the Higmann--Thompson groups obtained as a certain topological full group and the KMS states of their reduced crossed products.

\subsection*{Cuntz--Toeplitz analogue of $V_n$}

The Cuntz--Toeplitz algebra naturally appears in the theory of extensions of Cuntz algebras.
The generators $\{S_i\}_{i=1}^n$ of the Cuntz algebra $\mathcal{O}_n$ come from the creation operators on the Fock space.
These creation operators satisfy the Cuntz relation modulo compact operators, and the algebra generated by these creation operators, called the Cuntz--Toeplitz algebra $\mathcal{E}_n$, appears in the following extension
\[0\to \mathbb
K\to \mathcal{E}_n\to\mathcal{O}_n\to 0\]
where $\mathbb{K}$ denotes the algebra of compact operators on the Fock space (see Sec. \ref{cte}).
As in the case of $\mathcal{O}_n$,
we have a graph $E_n$ and its graph groupoid $\mathcal{G}_{E_n}$ realising $\mathcal{E}_n=C^*_r(\mathcal{G}_{E_n})$.
We will see that the unit space $\mathcal{G}_{E_n}^{(0)}$ of $\mathcal{G}_{E_n}$ is the union of the rooted $n$-regular tree $E^f_n$ with the root $v_0$ and the unit space $\mathcal{G}_{O_n}^{(0)}$.
This decomposition respects the above exact sequence of  Cuntz--Toeplitz extension.
In fact the subgroupoid of $\mathcal{G}_{E_n}$ obtained from $E^f_n$ is equal to $E^f_n\times E^f_n$ and one has $\mathbb{K}=C^*_r(E^f_n\times E^f_n)$.
These groupoids $E^f_n\times E^f_n, \; \mathcal{G}_{E_n}, \; \mathcal{G}_{O_n}$ yields the following exact sequence of the topological full groups
\[1\to [[E^f_n\times E^f_n]]\xrightarrow{i} [[\mathcal{G}_{E_n}]]\xrightarrow{\pi} V_n\to 1.\]
Note that the topological full group $[[E^f_n\times E^f_n]]$ is naturally identified with the group $\mathfrak{S}_{E^f_n}$ of finite permutations of the vertices of the tree $E^f_n$.

We study the abelianization $[[\mathcal{G}_{E_n}]]^{ab}$ and normal subgroups of $[[\mathcal{G}_{E_n}]]$.
Combining X. Li's result \cite{XinLi} and the simplicity of the commutator $V_n'$,
we prove the following (see Sec. \ref{main1}).

\begin{thm}[{Thm. \ref{ns}, Prop. \ref{abg}}]
\begin{enumerate}
\item The abelianization of $[[\mathcal{G}_{E_n}]]$ is $\mathbb{Z}/2\mathbb{Z}$ and computed as follows
\[i^{ab} \colon \mathfrak{S}_{E^f_{2n}}^{ab}\cong [[\mathcal{G}_{E_{2n}}]]^{ab},\quad \pi^{ab} \cong [[\mathcal{G}_{E_{2n+1}}]]^{ab}\cong V_{2n+1}^{ab},\]
where we write $f^{ab} \colon G^{ab}\to H^{ab}$ for the naturally induced map from $f \colon G\to H$.
\item For $n < \infty$,
    the non-trivial normal subgroups of $[[\mathcal{G}_{E_n}]]$ are $[[\mathcal{G}_{E_n}]]',\; [[R_\mathbb{N}]],\; [[R_\mathbb{N}]]'$,
    and the abelianization $[[\mathcal{G}_{E_n}]]^{ab}$ is $\mathbb{Z}/2\mathbb{Z}$.
    \end{enumerate}
\end{thm}
We note that the groupoids $\mathcal{G}_{E_n}$ are not minimal and their topological full groups $[[\mathcal{G}_{E_n}]]$ are not C*-simple and have normal subgroups other than the commutator subgroup.
There would be no other previous research on the topological full groups of the groupoids like $\mathcal{G}_{E_n}$. 

\subsection*{The gauge action and KMS states of $C(\mathcal{G}^{(0)})\rtimes_r[[\mathcal{G}]]$ analogus to that of $C^*_r(\mathcal{G})$}

The previous research on the gauge actions of $\mathcal{O}_n$ and $\mathcal{E}_n$ are generalized to the reseach on  the groupoid C*-algebra $C^*_r(\mathcal{G})$ and the $\mathbb{R}$-action $\gamma_c \colon \mathbb{R}\curvearrowright C^*_r(\mathcal{G})$ obtained from a cocycle $c \colon \mathcal{G}\to \mathbb{Z}$.
From the general theory of S. Neshveyev,
the $\gamma_c\textrm{-KMS}_\beta$-state is obtained by integrating a mesurable field of traces $\{\tau_x \colon C^*(\mathcal{G}_x^x)\to \mathbb{C}\}_{x\in \mathcal{G}^{(0)}}$ via $\beta$-conformal measure $m \colon C(\mathcal{G}^{(0)})\to \mathbb{C}$ (see \cite{Nes, VC}).
If the stabilizer $\mathcal{G}_x^x$ at $x\in\mathcal{G}^{(0)}$ is trivial (or the trace on "$C^*(\mathcal{G}_x^x)$" is trivial),
this boils down to the composition
\[C^*(\mathcal{G})\to C(\mathcal{G}^{(0)})\xrightarrow{m}\mathbb{C}\]
(see Thm. \ref{M}), where $C^*(\mathcal{G})\to C(\mathcal{G}^{(0)})$ is the conditional expectation.
In general,
the structure of the measurable field $\{\tau_x\}_{x\in \mathcal{G}^{(0)}}$ is complicated.

In this paper,
we forcus on the transformation groupoid $\mathcal{G}^{(0)}\rtimes [[\mathcal{G}]]$ (i.e., the crossed product $C(\mathcal{G}^{(0)})\rtimes_r [[\mathcal{G}]]$) and study the $\mathbb{R}$-action $\gamma_{c^f} \colon \mathbb{R}\curvearrowright C(\mathcal{G}^{(0)})\rtimes_r[[\mathcal{G}]]$ induced by the cocycle
\[c^f \colon \mathcal{G}^{(0)}\rtimes [[\mathcal{G}]]\to \mathcal{G}\xrightarrow{c}\mathbb{Z}.\]
The canonical gauge actions of $\mathcal{O}_n$ and $\mathcal{E}_n$ come from the canonical cocycle of the Deaconu--Renault groupoids
\[c_n \colon \mathcal{G}_{O_n}\to \mathbb{Z},\quad d_n \colon \mathcal{G}_{E_n}\to\mathbb{Z},\]
and we investigate the KMS states of the $\mathbb{R}$-actions on $C(\mathcal{G}_{O_n}^{(0)})\rtimes_r[[\mathcal{G}_{O_n}]]$ and $C(\mathcal{G}_{E_n}^{(0)})\rtimes_r[[\mathcal{G}_{E_n}]]$ induced by the cocycle $c^f_n$ and $d^f_n$.
\begin{thm}[see Thm. \ref{M}, \ref{MM}, \ref{MMM}, \ref{kms}]
\begin{enumerate}
\item There is a $\gamma_{c_n^f}$-KMS$_n$-state on $C(\mathcal{G}^{(0)}_{O_n}) \rtimes _r [[ \mathcal{G}_{O_n} ]]$ if and only if $\beta = \log n$.
There is a unique $\gamma_{c_n^f}$-KMS$_{\log n}$-state given by 
\[
C(\mathcal{G}^{(0)}_{O_n}) \rtimes _r [[ \mathcal{G}_{O_n} ]] \overset{E}{\to} C(\mathcal{G}_{O_n}^{(0)}) \overset{m}{\to} \mathbb{C}
\]
where $E$ is the canonical conditional expectation and $m$ is the product measure $\bigotimes_{k=1}^{\infty} (\sum _{j=1}^n \frac{1}{n} \delta _j)$ on $\mathcal{G}^{(0)}_{O_n} \cong \{ 1, 2, \ldots , n \}^{\infty}$.
There is no $\gamma _{c_n^f}$-ground state on $C(\mathcal{G}^{(0)}_{O_n}) \rtimes _r [[ \mathcal{G}_{O_n} ]]$.
\item There is no $\gamma _{d_{\infty}^f}$-KMS$_\beta$-state on $C(\mathcal{G}_{E_{\infty}}^{(0)}) \rtimes _r [[ \mathcal{G}_{E_\infty} ]]$ for $0 \leq \beta < \infty$.
There is a one-to-one correspondence between the $\gamma _{d_{\infty}^f}$-ground states and the states of $C_r^*([[ \mathcal{G}_{E_\infty} ]]_{v_0})$ where $v_0 \in \mathcal{G}_{E_\infty}^{(0)}$ is an element corresponds to the root of rooted $\infty$-regular tree.
\item  There is a $\gamma _{d_n^f}$-KMS$_{\beta}$-state on $C(\mathcal{G}_{E_n}^{(0)}) \rtimes _r [[ \mathcal{G}_{E_n}]]$ if and only if $\beta \geq \log n$.
For $\beta=\log n$, there is a one-to-one correspondence between $\gamma_{d_n^f}$-KMS$_{\log n}$-states and the tracial states of $C_r^*([[ R_\mathbb{N} ]])$.
For $\beta > \log n$, there is a one-to-one correspondence between $\gamma_{d_n^f}$-KMS$_{\beta}$-states and the tracial states of $C_r^*([[ \mathcal{G}_{E_n}]] _{v_0})$ where $v_0 \in \mathcal{G}_{E_n}^{(0)}$ is an element corresponds to the root of rooted $n$-regular tree.
There is a one-to-one correspondence between the $\gamma _{d_n^f}$-ground states on $C(\mathcal{G}_{E_n}^{(0)}) \rtimes _r [[ \mathcal{G}_{E_n}]]$ and the states of $C_r^*([[ \mathcal{G}_{E_n}]] _{v_0})$.
\end{enumerate}
\end{thm}
What is interesting is that the structure of measurable field $\{\tau_x\}_{x\in \mathcal{G}^{(0)}}$ in the S. Neshveyev's picture tends to be easy in spite of the huge stabilizer $(\mathcal{G}^{(0)}\rtimes [[\mathcal{G}]])_x^x$.
This happens because of the unique trace property of $V_n$ by which the traces $\tau_x \colon C^*_r((V_n)_x)\to\mathbb{C}$ on the stabilizers are all canonical traces and the traces on the stabilizer $(\mathcal{G}_{E_n}^{(0)}\rtimes [[\mathcal{G}_{E_n}]])_x^x$ all factor through \[(\mathcal{G}_{E_n}^{(0)}\rtimes \mathfrak{S}_{E^f_n})_x^x=\mathfrak{S}_{E^f_n}.\]
This reads us to conclude that S. Neshveyev's measurable field comes from  a trace on \[\int_{\mathcal{G}_{E_n}^{(0)}} \tau_x dm(x) \colon C(\mathcal{G}_{E_n})\rtimes_r \mathfrak{S}_{E^f_n}=C(\mathcal{G}_{E_n}^{(0)})\otimes C^*_r(\mathfrak{S}_{E^f_n})\to \mathbb{C}.\]
Furthermore,
we show that the centralizer of $h\in \mathfrak{S}_{E^f_n}$ in $[[\mathcal{G}_{E_n}]]$ surjects to $V_n$ by which we can conclude the measurable field is constant (i.e., $\int_{\mathcal{G}_{E_n}^{(0)}} \tau_x dm(x)=m\otimes \tau$).

As explained above,
the unique trace property of $V_n$ plays an important role.
For $\mathcal{G}=\mathcal{G}_{O_n}, \mathcal{G}_{E_n}$, the groupoid $\mathcal{G}^{(0)}\rtimes[[\mathcal{G}]]$ tends to have the large isotropies and each isotropy tends to have fewer traces by the unique trace property that enables us to compute the KMS states by using the strategy of  \cite{Nes, P}.

\section*{Acknowledgement}
We would like to thank Fuyuta Komura, Yosuke Kubota, and Takuya Takeishi for their helpful comments and discussions.
T. Sogabe is supported by JSPS KAKENHI Grant Number 24K16934.

%%%%%%%%%%%%%%%%%%%%%%%%%%%%%%%%%%%%%%%%%%%%%%%%%%%%%%%%%%%%%%%%%%%%%%%%%%%%%%%%%%%%%%%%%%%%%%%%%%%%%%%%%%%%%%%%%%%%%%%%%%%%%%%%%%%%%%%%%%%%%%%%%%%%%%%%%%%%%%%%%%%%%%%%%%%%%%%%%%%%%%%%%%%%%%%%%%%%%%%%%%%%%%%%%%%%%%%%%%%%%%%%%%%%%%%%%%%%%%%%%%%%%%%%%%%%%%%%%%%%%%%%%%%%%%%%%%%%%%%%%%%%%%%%%%%%%%%%%%%%

\section{Preliminaries}\label{s2}

\subsection{Cuntz algebras and Cuntz--Toeplitz algebras}\label{cte}
\begin{dfn}[{cf. \cite{Cu}}]
    For $2 \leq n < \infty$, the Cuntz algebra $\mathcal{O}_n$ is the universal C*-algebra generated by the isometries $S_1,\cdots, S_n$ with mutually orthogonal ranges (i.e., $S_i^*S_j=\delta_{i, j} 1_{\mathcal{O}_n}$) satisfying
    \[S_1S_1^*+\cdots +S_nS_n^*=1_{\mathcal{O}_n},\]
    where $1_{\mathcal{O}_n}$ is the unit of $\mathcal{O}_n$.
    
    For $2 \leq n < \infty$, the Cuntz--Toeplitz algebra $\mathcal{E}_n$ is the universal C*-algebra generated by the isometries $T_1, \cdots, T_n$ with mutually orthogonal ranges (i.e., $T_i^*T_j=\delta_{i, j} 1_{\mathcal{E}_n}$).
    
    The infinite Cuntz algebra $\mathcal{O}_\infty$ is the universal C*-algebra generated by the isometries $\{T_i\}_{i=1}^\infty$ with mutually orthogonal ranges (i.e., $T_i^*T_j=\delta_{i, j} 1_{\mathcal{O}_\infty}$). 
\end{dfn}
We write \[e_n:=1_{\mathcal{E}_n}-\sum_{i=1}^nT_iT_i^*\in\mathcal{E}_n.\]
By the universality,
one has the following $*$-homomorphisms
\[\pi \colon \mathcal{E}_n\ni T_i\mapsto S_i\in \mathcal{O}_n,\]
\[\mathcal{E}_n\ni T_i\mapsto T_i\in \mathcal{O}_\infty, \quad \mathcal{O}_\infty=\overline{\bigcup_{n=2}^\infty\mathcal{E}_n}.\]
Since $T_{\mu_1}\cdots T_{\mu_s} e_n T_{\nu_t}^*\cdots T_{\nu_1}^*, \; \mu_i, \nu_j\in \{1, \cdots, n\}$ provides a matrix unit,
the ideal of $\mathcal{E}_n$ generated by $e_n$ is isomorphic to $\mathbb{K}$, and this is contained in $\operatorname{Ker}\pi$ because $\pi (e_n)=1_{\mathcal{O}_n}-\sum_{i=1}^n S_iS_i^*=0$.
For the map
\[\mathcal{E}_n/\mathbb{K}\to\mathcal{E}_n/\operatorname{Ker}\pi=\mathcal{O}_n,\]
the simplicity (see \cite{Cu}) and universality of the Cuntz algebra give the inverse \[\mathcal{O}_n\ni S_i\mapsto \bar{T}_i\in\mathcal{E}_n/\mathbb{K}\] which implies $\mathbb{K}=\operatorname{Ker}\pi$,
and one has the extension
\[0\to\mathbb{K}\to\mathcal{E}_n\xrightarrow{\pi}\mathcal{O}_n\to 0.\]

The Cuntz--Toeplitz algebra $\mathcal{E}_n$ has the Fock representation
\[\mathcal{E}_n \subset \mathbb{B}(\mathcal{F}(\mathbb{C}^n)), \quad \mathcal{F}(\mathbb{C}^n):=\mathbb{C}\oplus\bigoplus_{k=1}^\infty (\mathbb{C}^n)^{\otimes k},\]
\[T_i \colon \mathbb{C}\ni z\mapsto z\delta_i\in \mathbb{C}^n,\;\; (\delta_i : \text{$i$-th orthogonal basis}),\]
\[T_i \colon (\mathbb{C}^n)^{\otimes k}\ni \zeta\mapsto \delta_i\otimes\zeta\in (\mathbb{C}^n)^{\otimes k+1}.\]
We will identify $\mathcal{F}(\mathbb{C}^n)$ with the Hilbert space $\ell^2(\{v_0\}\cup\bigcup_{k=1}^\infty\{1, \cdots, n\}^k)$ whose orthogonal basis are denoted by \[\{\delta_\mu\}_{_\mu\in \{v_0\}\cup\bigcup_{k=1}^\infty\{1,\cdots, n\}^k}.\]
The definition of \[\mathcal{F}(\ell^2(\mathbb{N}))=\ell^2\left(\{v_0\}\cup\bigcup_{k=1}^\infty\mathbb{N}^k\right)\] is the same as above and this provides the Fock representation of $\mathcal{O}_\infty$.

%%%%%%%%%%%%%%%%%%%%%%%%%%%%%%%%%%%%%%%%%%%%%%%%%%%%%%%%%%%%%%%%%%%%%%%%%%%%%%%%%%%%%%%%%%%%%%%%%%%%%%%%%%%%%%%%%%%%%%%%%%%%%%%%%%%%%%%%%%%%%%%%%%%%%%%%%%%%%%%%%%%%%%%%%%%%%%%%%%%%%%%%%%%%%%%%%%%%%%%%%%%%%%%%%%%%%%%%%%%%%%%%%%%%%%%%%%%%%%%%%%%%%%%%%%%%%%%%%%%%%%%%%%%%%%%%%%%%%%%%%%%%%%%%%%%%%%%%%%%%%%%%%%%%%%%%%%%%%%%%%%%%%%%%%%%%%%

\subsection{\'{E}tale groupoids and topological full groups}\label{fgr}
We refer \cite{REN} for the basics of the groupoids and their C*-algebras.
We also refer to \cite[Sec. 2.3., Sec. 3.]{NO} for the \'{e}tale groupoids and their topological full groups.
A topological groupoid is a pair of topological spaces $\mathcal{G}^{(0)}\subset\mathcal{G}$ with the following continuous structure maps
\[\text{Range and Source maps}\quad r, s \colon \mathcal{G}\to\mathcal{G}^{(0)},\quad r(x)=s(x)=x, \; x\in \mathcal{G}^{(0)},\]
\[\text{Associative multiplication map}\quad \mathcal{G}^{(2)}:=\{(g, h)\in \mathcal{G}\times\mathcal{G}\;|\; s(g)=r(h)\}\ni (g, h)\mapsto gh\in\mathcal{G},\]
\[ g=r(g)g=gs(g),\quad r(gh)=r(g),\quad s(gh)=s(h),\]
\[\text{Inverse map} \;\; \mathcal{G}\ni g\mapsto g^{-1}\in\mathcal{G},\quad gg^{-1}=r(g),\quad g^{-1}g=s(g).\]
We say the groupoid $\mathcal{G}$ is locally compact Hausdorff if so is the topological space $\mathcal{G}$.
The groupoid $\mathcal{G}$ is called minimal, if $\{r(g)\in \mathcal{G}^{(0)} \mid g\in\mathcal{G}, s(g)=x\}$ is a  dense subset of $\mathcal{G}^{(0)}$ for every $x\in\mathcal{G}^{(0)}$.
A topological groupoid is called \'{e}tale if the range and source maps are local homeomorphisms.
Note that $\mathcal{G}^{(0)}$ is an open subset of $\mathcal{G}$ for any  \'{e}tale groupoid.
An open subset $U\subset \mathcal{G}$ is called open bisection if $r|_U \colon U\to r(U)$ and $s|_U \colon U\to s(U)$ are homeomorphisms,
and an \'{e}tale groupoid has the open basis consisting of open bisections.
An \'{e}tale groupoid is called ample if $\mathcal{G}^{(0)}$ is Hausdorff and has an open basis consisting of compact open sets.
Note that the locally compact Hausdorff space is totally disconnected if and only if it has an open basis consisting of clopen sets.
We basically consider locally compact Hausdorff, ample groupoids.

A bisection $U\subset\mathcal{G}$ is called full bisection,
if $s(U)=r(U)=\mathcal{G}^{(0)}$.
For two full bisections $U, V\subset\mathcal{G}$,
$UV:=\{uv\in\mathcal{G}\;|\; u\in U, v\in V, \;(u, v)\in\mathcal{G}^{(2)}\}$ is also a full bisection.
Let $\operatorname{supp}U:=\{x\in \mathcal{G}^{(0)}\;|\; s^{-1}(x)\cap U\not=\{x\}\}$, then $\operatorname{supp}UV$ and $\operatorname{supp}U^{-1}$ are compact if $\operatorname{supp}U$ and $\operatorname{supp}V$ are compact.
Thus,  the set of full bisection with compact support, denoted by $[[\mathcal{G}]]$, is a group where the unit is $\mathcal{G}^{(0)}$ and the inverse of $U$ is given by $U^{-1}:=\{u^{-1}\in\mathcal{G}\;|\; u\in U\}$.
Since $U\in[[\mathcal{G}]]$ defines a homeomorphism $(s(U)\xrightarrow{r|_U\circ s|_U^{-1}} r(U))\in\operatorname{Homeo}(\mathcal{G}^{(0)})$,
there is a group homomorphism $[[\mathcal{G}]]\to\operatorname{Homeo}(\mathcal{G}^{(0)})$.

A groupoid $\mathcal{G}$ is topologically principal if the set 
\[\{x\in\mathcal{G}^{(0)}\;|\; \{g\in\mathcal{G}\;|\; r(g)=s(g)=x\}=\{x\}\}\]
is dense in $\mathcal{G}^{(0)}$.
A groupoid $\mathcal{G}$ is called essentially principal (or effective) if the interior of the set
\[\{g\in\mathcal{G}\;|\; r(g)=s(g)\}\]
is equal to $\mathcal{G}^{(0)}$.
By \cite[Prop. 3.1.]{REN2},
topological principality and essential principality (effectiveness) are equal for the second countable \'{e}tale groupoids.

For a topologically principal, locally compact Hausdorff, ample groupoid $\mathcal{G}$,
the map $[[\mathcal{G}]]\to\operatorname{Homeo}(\mathcal{G}^{(0)})$ is injective (see \cite[Lem. 3.1.]{NO}).
So, we will identify $[[\mathcal{G}]]$ with the subgroup of $\operatorname{Homeo}(\mathcal{G}^{(0)})$ frequently in the subsequent sections,
and write $U (x):=r|_U\circ s|_U^{-1}(x),\; x\in \mathcal{G}^{(0)}$ for short.

The multiplication of $\mathcal{G}$ induces the convolution product of the set $C_c(\mathcal{G})$ of the $\mathbb{C}$-valued, compactly supported, continuous functions on $\mathcal{G}$ by
\[1_U\cdot 1_V:=1_{UV},\]
where $U, V$ are clopen bisections and $1_U$ is the characteristic function of $U$.
If $\mathcal{G}^{(0)}$ is compact $1_{\mathcal{G}^{(0)}}$ is the unit of $C_c(\mathcal{G})$,
and we have a multiplicative injection
\[[[\mathcal{G}]]\ni U\mapsto 1_U\in C_c(\mathcal{G})\]
(see Lem. \ref{matui}).

By taking appropriate completion,
one obtains the reduced (resp. full) groupoid C*-algebra $C^*_r(\mathcal{G})$ (resp. $C^*(\mathcal{G})$) (see \cite{REN}).
The involution of this C*-algebra is given by $(1_U)^*:= 1_{U^{-1}}$.
There is a distinguished subalgebra
\[C(\mathcal{G}^{(0)})=\overline{\{f\in C_c(\mathcal{G})\;|\; \operatorname{supp}f\subset \mathcal{G}^{(0)}\}},\]
and the adjoint action of $U\in [[\mathcal{G}]]$ satisfies
\[\operatorname{Ad}1_U(f):=1_U\cdot f\cdot (1_U)^*\in C(\mathcal{G}^{(0)}), \quad 1_U\cdot f\cdot (1_U)^*(x)=f(U^{-1}(x)),\;\; x\in \mathcal{G}^{(0)}.\]

For a unital C*-algebra $A$,
we write $U(A):=\{u\in A\;|\; uu^*=1=u^*u\}$.
The topological full group $[[\mathcal{G}]]$ is understood as a subgroup of the unitary group $[[\mathcal{G}]]\subset U(C^*_r(\mathcal{G}))$ if $\mathcal{G}^{(0)}$ is compact.

\begin{ex}
    For a countable set $F$ with the discrete topology, we write \[R_F:=F\times F, R_F^{(0)}:=\{(m, m)\in R_F\}= F,\quad   r(m, n):=m, s(m, n):=n,\]  \[R_F^{(2)}:=\{((m, n), (n, k))\in R_F\times R_F\}, \quad (m, n)(n, l)=(m, l),\quad (m, n)^{-1}:=(n, m).\]
We denote by $\mathfrak{S}_F$ the group of finite permutations of $F$,
and this is naturally identified with $[[R_F]]$.
For the convolution algebra $C_c(R_F)$,
the elements $\{1_{(m, n)}\}_{m, n\in F}$ satisfing
\[1_{(m, n)}\cdot 1_{(l, k)}=\delta_{n, l}1_{(m, k)}\]
form a matrix unit,
and one has $C^*_r(R_F)=C^*(R_F)=\mathbb{K}(\ell^2(F))$ where $\mathbb{K}(\ell^2(F))$ is the algebra of compact operators on the Hilbert space $\ell^2(F)$.
One has $C_0(R_F^{(0)})=c_0(F)\subset \ell^\infty(F)$.
If $F$ is a finite set,
this gives an embedding of the symmetric group $\mathfrak{S}_F$ into the algebra of $|F| \times |F|$ matrices $\mathbb{M}_{|F|}(\mathbb{C})=C^*_r(R_F)$.

Note that, for $F=\mathbb{N}$, $R_\mathbb{N}^{(0)}$ is not compact (i.e., $C^*_r(R_\mathbb{N})=\mathbb{K}$ is non-unital), and $[[R_\mathbb{N}]]=\mathfrak{S}_\mathbb{N}$ is a subset of $1+\mathbb{K}$.
\end{ex}

\begin{lem}[{\cite[Prop. 5.6.]{matui}}]\label{matui}
    Let $\mathcal{G}$ be a topologically principal, second countable, \'{e}tale groupoid whose unit space $\mathcal{G}^{(0)}$ is compact.
    Let
    \[N(C^*_r(\mathcal{G}), C(\mathcal{G}^{(0)})):=\{u\in U(C^*_r(\mathcal{G}))\;|\; uC(\mathcal{G}^{(0)})u^*=C(\mathcal{G}^{(0)})\}\]
    be the normalizer of $C(\mathcal{G}^{(0)})$ in $\mathrm{C}^*_r(\mathcal{G})$.
    \begin{enumerate}
        \item There is an exact sequence
        \[1\to U(C(\mathcal{G}^{(0)}))\to N(C^*_r(\mathcal{G}), C(\mathcal{G}^{(0)}))\xrightarrow{\sigma} [[\mathcal{G}]]\to 1.\]
        \item The above map $N(C^*_r(\mathcal{G}), C(\mathcal{G}^{(0)}))\to [[\mathcal{G}]]$ is induced by identifying the adjoint action $\operatorname{Ad} u \in\operatorname{Aut}(C(\mathcal{G}^{(0)}))$ with a homeomorphism $\sigma(u)\in \operatorname{Homeo}(\mathcal{G}^{(0)})$ as follows:
       \[  f(\sigma(u)^{-1}(x)) := ufu^*(x), \quad f\in C(\mathcal{G}^{(0)}), \;\; u\in N(C^*_r(\mathcal{G}), C(\mathcal{G}^{(0)})).\]        
        \item The map $\sigma$ has a splitting \[[[\mathcal{G}]]\ni U\mapsto 1_U\in N(C^*_r(\mathcal{G}), C(\mathcal{G}^{(0)}))\] and $\sigma (1_U)=U, \;U\in [[\mathcal{G}]]$ holds.
    \end{enumerate}
    \end{lem}

%%%%%%%%%%%%%%%%%%%%%%%%%%%%%%%%%%%%%%%%%%%%%%%%%%%%%%%%%%%%%%%%%%%%%%%%%%%%%%%%%%%%%%%%%%%%%%%%%%%%%%%%%%%%%%%%%%%%%%%%%%%%%%%%%%%%%%%%%%%%%%%%%%%%%%%%%%%%%%%%%%%%%%%%%%%%%%%%%%%%%%%%%%%%%%%%%%%%%%%%%%%%%%%%%%%%%%%%%%%%%%%%%%%%%%%%%%%%%%%%%%%%%%%%%%%%%%%%%%%%%%%%%%%%%%%%%%%%%%%%%%%%%%%%%%%%%%%%%%%%%%%%%%%%%%%%%%%%%%%%%%%%%%%%%%%%%%

\subsection{Graph groupoids for $\mathcal{O}_n$}
We refer to \cite[Sec. 8]{NO} for the basics and terminologies of the graphs, their boundary path spaces, and the graph groupoids.

Let $O_n:=(O_n^0, O_n^1, r, s)$ be the following graph
\[O_n^0:=\{V\}, \quad, O_n^1:=\{e_1,\cdots, e_n\},\quad r(e_i)=s(e_i)=V.\]

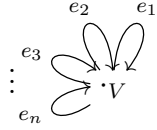
\[   \xymatrix{\vdots&\ar@(u, ur)[]^{e_1} \ar@(ul,
u)[]^{e_2} \ar@(l, lu)[]^{e_3} \ar@(ld, l)[]^{e_n} {\cdot}_V }
\]

Since $(O_n^0)_{\text{sing}}=\emptyset$, the boundary path space (see \cite[Sec. 8.1.]{NO}) is identified with
\[\partial O_n=O_n^\infty\ni e_{\mu_1}e_{\mu_2}\cdots\; \mapsto \mu_1\mu_2\cdots\; \in \{1, \cdots, n\}^\infty.\]
The shift map $\sigma_{O_n} \colon \partial O_n\ni \mu_1\mu_2\mu_3\cdots\;\mapsto \mu_2\mu_3\cdots\;\in\partial O_n$ defines the following Deaconu--Renault groupoid
\[\mathcal{G}_{O_n}:=\{(x, k-l, y)\in\partial O_n\times\mathbb{Z}\times \partial O_n\;|\; \sigma_{O_n}^k(x)=\sigma_{O_n}^l(y),\; k, l\in\mathbb{Z}_{\geq 0}\},\]
with the base space $\mathcal{G}^{(0)}_{O_n}:=\{(x, 0, x)\;|\; x\in\partial O_n\}=\partial O_n$
and the structure maps
\[s(x, k, y)=y,\; r(x, k, y)=x,\;\; (x, k, y)\cdot (y, l, z)=(x, k+l, z).\]
%As in the case of $\mathcal{G}_{E_n}$,
A clopen basis of $\partial O_n$ is given by the cylinder set
\[Z^\infty(\mu):=\{\mu x\in\partial O_n\;|\; x\in\partial O_n\}, \quad \mu \in \{v_0\}\cup\bigcup _{k \geq 1} \{ 1, \ldots , n \} ^k.\]
The groupoid $\mathcal{G}_{O_n}$ becomes a locally compact, second countable, topologically principal, ample groupoid via the following open basis
\[Z(Z^\infty(\mu), |\mu|, |\nu|, Z^\infty(\nu)):=\{(\mu x, |\mu|-|\nu|, \nu x)\; |\; x\in \partial O_n\}, \quad \mu, \nu \in \{v_0\}\cup\bigcup _{k \geq 1} \{ 1, \ldots , n \} ^k,\]
where we write $|\mu|=|\mu_1\cdots\mu_k|=k$ for $\mu_i\in \{1, \cdots, n\}$.
Note that we use the convention $v_0 x=x,\; |v_0|=0,\;\; Z^\infty(v_0)=\partial O_n$.

By the universality of Cuntz relation,
one has a map
\[\mathcal{O}_n\ni S_i\mapsto 1_{Z(Z^\infty(i), 1, 0, Z^\infty(v_0))}\in C^*_r(\mathcal{G}_{O_n})\]
which is surjective because $C_r^*(\mathcal{G}_{O_n})$ is generated by the characteristic function of clopen bisection $1_{Z(Z^\infty(\mu), |\mu|, |\nu|, Z^\infty(\nu))}$ which  is the image of $S_\mu S_{\nu}^*$.
We use the convention $S_{v_0}=1$.
Since $\mathcal{O}_n$ is simple, the above map is an isomorphism.

For the graph $O_n$,
the graph C*-algebra is given by the following universal C*-algebra
\[C^*(O_n):=C^*_{\text{univ}}(\{S_{e_i}, P_V \mid P_V^2 = P_V = P_V^*, \; S^*_{e_i}S_{e_i}=P_V,\; \sum_{i=1}^n S_{e_i}S_{e_i}^*=P_V\}),\]
and the isomorphism
\[\mathcal{O}_n\ni S_i\mapsto S_{e_i}\in C^*(O_n)\]
follows by definition.
The Cartan subalgebra $C(\mathcal{G}_{O_n}^{(0)})=C(\partial O_n)$ is identified with \[\overline{\operatorname{span}} \{1, S_\mu S_\mu^*\; |\; \mu\in \bigcup_{k=1}^\infty\{1, \cdots, n\}^k\}, \;\; n<\infty.\]

%%%%%%%%%%%%%%%%%%%%%%%%%%%%%%%%%%%%%%%%%%%%%%%%%%%%%%%%%%%%%%%%%%%%%%%%%%%%%%%%%%%%%%%%%%%%%%%%%%%%%%%%%%%%%%%%%%%%%%%%%%%%%%%%%%%%%%%%%%%%%%%%%%%%%%%%%%%%%%%%%%%%%%%%%%%%%%%%%%%%%%%%%%%%%%%%%%%%%%%%%%%%%%%%%%%%%%%%%%%%%%%%%%%%%%%%%%%%%%%%%%%%%%%%%%%%%%%%%%%%%%%%%%%%%%%%%%%%%%%%%%%%%%%%%%%%%%%%%%%%%%%%%%%%%%%%%%%%%%%%%%%%%%%%%%%%%%

\subsection{Boundary path spaces of graphs for $\mathcal{E}_n$ and $\mathcal{O}_\infty$}\label{not}
For $2 \leq n < \infty$, consider the following graph $E_n:=(E_n^0, E_n^1, r, s)$ such that $E_n^0=\{V_0, V\}$, $E^1_n=\{e_1,\cdots, e_n\}\sqcup\{f_1,\cdots, f_n\}$ with
\[r(e_i)=s(e_i)=V,\quad s(f_i)=V,\quad r(f_i)=V_0,\;\;\text{for}\;\; i\in\{1, \cdots, n\}.\]
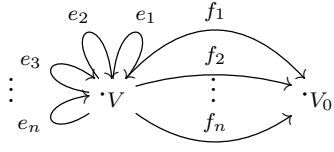
\[   \xymatrix{\vdots&\ar@(u, ur)[]^{e_1} \ar@(ul,
u)[]^{e_2} \ar@(l, lu)[]^{e_3} \ar@(ld, l)[]^{e_n} {\cdot}_V \ar@/^0.6pc/[rr]^{f_2} \ar@/^2.0pc/[rr]^{f_1}
\ar@/_1.5pc/[rr]^{f_n} & \vdots& \cdot _{V_0}}
\]
Then, the boundary path space $\partial E_n$ (see \cite[Sec. 8.1.]{NO}) is defined by
\[\partial E_n:=\{e_{\mu_1}e_{\mu_2}\cdots \;|\; \text{infinite path}\}\cup \{e_{\mu_1}e_{\mu_2}\cdots e_{\mu_{k-1}}f_{\mu_k} \;|\; \text{finite path with the range}\; V_0\}\cup \{V_0\}.\]
In this paper, we identify $\partial E_n$ with the following vertex set of rooted $n$-regular tree with the root $v_0$
\[\{v_0\}\cup\bigcup_{k=1}^\infty\{1, \cdots, n\}^k\cup \{1, \cdots, n\}^\infty\]by the correspondence
\[V_0\mapsto v_0,\] \[\{\text{finite paths with the range}\; V_0\}\ni e_{\mu_1}e_{\mu_2}\cdots e_{\mu_{k-1}}f_{\mu_k}\mapsto \mu_1\cdots\mu_k\in \bigcup_{k=1}^\infty\{1,\cdots, n\}^k,\] \[E_n^\infty:=\{\text{infinite paths}\} \ni e_{\mu_1}e_{\mu_2}\cdots \mapsto \mu_1\mu_2\cdots\in \{1, \cdots, n\}^\infty.\]

For $n=\infty$,
we consider the graph
$E_\infty:=(E^0_\infty, E^1_\infty, r, s)$ such that
    \[E^0_\infty=(E^0_\infty)_{\text{sing}}=\{V_0\}, \quad E^1_\infty=\{e_i\}_{i=1}^\infty,\quad r(e_i)=s(e_i)=V_0.\]
    
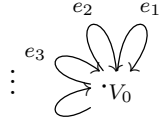
\[   \xymatrix{\vdots&\ar@(u, ur)[]^{e_1} \ar@(ul,
u)[]^{e_2} \ar@(l, lu)[]^{e_3} \ar@(ld, l)[] {\cdot}_{V_0} }
\]
    Then, the boundary path space is given by
    \[\partial E_\infty:=\{e_{\mu_1}e_{\mu_2}\cdots \;|\;\text{infinite paths}\}\cup \{\text{finite paths}\}\cup \{V_0\},\quad E^\infty_\infty:=\{\text{infinite paths}\}\]
and naturally identified with $\{v_0\}\cup\bigcup_{k=1}^\infty\mathbb{N}^k\cup\mathbb{N}^\infty$.

%From here, the notation $E_n, \{1, \cdots, n\}$ are used for $2 \leq n \leq \infty$ (i.e., $\{1, \cdots, \infty\}:=\mathbb{N}$).
%For $\infty\geq n\geq 2$,
We write 
\[E^f_n:=\partial E_n\backslash E_n^\infty=\{v_0\}\cup\bigcup_{k=1}^\infty \{1, \cdots, n\}^k,\quad 2 \leq n \leq \infty.\]

Following \cite[Sec.8.1.]{NO},
we put a topology on $\partial E_n$.
For a finite path $\mu_1\mu_2\cdots\mu_k\in E^f_n, \mu_i\in \{1, \cdots, n\}$,
we write $|\mu|:=k$, $|v_0|:=0$.
We write
\[Z(\mu):=\{\mu x\in \partial E_n\;|\; x\in \partial E_n\}, \quad Z^\infty(\mu):=Z(\mu)\cap E^\infty_n,\quad Z(v_0):=\partial E_n.\]
%Following \cite[Sec. 5.2. pp179--181]{BO},
The open basis are given by the following sets
\[Z(\mu),\quad Z(\mu)\backslash \cup_{i=1}^N Z(\mu \nu_i),\quad \mu, \nu_i\in E^f_n,\;\; N\in\mathbb{N}.\]
Then, $\partial E_n$ is totally disconnected, compact Hausdorff space with the above clopen basis.

For $2 \leq n<\infty$,
the open basis of $\partial E_n$ are given by the clopen sets
\[Z(\mu),\quad \{\mu\},\;\;\mu\in E^f_n.\]
\begin{remark}
    There is a slight notational difference between \cite{NO} and our paper in the case of $\partial E_n$ for $2 \leq n<\infty$.
    
    Our set $Z(\mu), \mu\in E^f_n$ corresponds to  $Z(e_{\mu_1}\cdots e_{\mu_k}),\;  e_{\mu_1}\cdots e_{\mu_k}\in E^*_n$ in \cite{NO}.
    In \cite{NO}, one also has $Z(e_{\mu_1}\cdots e_{\mu_{k-1}}f_{\mu_k})=\{e_{\mu_1}\cdots f_{\mu_k}\}=\{\mu\}$, and this is given by $\{\mu\}=Z(\mu)\backslash \sqcup_{k=1}^nZ(\mu k)$ in our setting.
    Thus, the topology of their boundary path space coincides with ours.
    The topology of $\partial E_n$ ($n=2,\cdots, \infty$) in this paper 
is also equal to the topology used in \cite[Sec. 5.2. pp179--181]{BO} to compactify the rooted $n$-regular tree ($2 \leq n \leq \infty$).
\end{remark}

\begin{remark}\label{r1}
For $n<\infty$,
    $E^f_n\subset \partial E_n$ is an open dense subset and the boundary $E^\infty_n$ is closed.
    On the other hand, $E^f_\infty\subset\partial E_\infty$ is dense but not open.
    % because $E^\infty_\infty$ is also dense in $\partial E_\infty$.
\end{remark}
The graph $E_n$ ($2 \leq n < \infty$) and $E_\infty$ define the following graph C*-algebras
\begin{eqnarray*}
C^*(E_n):=C^*_{\text{univ}} \left( \left\{ S_{e_i}, S_{f_i}, P_{V_0}, P_V 
\middle| 
\begin{gathered}
P_{V}^2 = P_{V} = P_{V}^*, P_{V_0}^2 = P_{V_0} = P_{V_0}^* \\
S^*_{e_i}S_{e_i}=P_V,\;\; S_{f_i}^*S_{f_i}=P_{V_0},\; \sum_i (S_{e_i}S_{e_i}^*+S_{f_i}S_{f_i}^*) =P_V
\end{gathered}
 \right\} \right), \\
C^*(E_\infty):=C^*_{\text{univ}}(\{S_{e_i}, P_{V_0} \mid P_{V_0}^2 = P_{V_0} = P_{V_0}^*, \; S_{e_i}^*S_{e_i}=P_{V_0}, \; \sum_{i=1}^NS_{e_i}S^*_{e_i}<P_{V_0}, \; \text{for any} \; N\in\mathbb{N}\}),
\end{eqnarray*}
and it is easy to check the isomorphism
\[\mathcal{O}_\infty\ni T_i\mapsto S_{e_i}\in C^*(E_\infty).\]
Since $\{S_{e_i}+S_{f_i}\}_{i=1}^n$ are isometries with mutually orthogonal ranges, there is a map
\[\mathcal{E}_n\ni T_i\mapsto S_{e_i}+S_{f_i}\in C^*(E_n).\]
This map is an isomorphism since it sends $e_n=1_{\mathcal{E}_n}-\sum_{i=1}^n T_iT_i^*$ to $P_{V_0}\not=0$ and $S_{f_i}=(S_{e_i}+S_{f_i})P_{V_0}$.

%Let $\mathcal{T}_n$ be the space
%\[\mathcal{T}_n:=\{v_0\}\cup\bigcup_{k=1}^\infty\{1, \cdots, n\}^k, \;\;\;\text{for}\;\; n=2, 3, \cdots, \infty.\]
%We write $\partial\mathcal{T}_n:=\{1, \cdots, n\}^\infty$ and $\overline{\mathcal{T}}_n:=\mathcal{T}_n\cup\partial\mathcal{T}_n$.
%We identify $\overline{\mathcal{T}}_n$ with the space of empty, finite, and infinite words of the alphabets $\{1, \cdots, n\}$.
%We also identify $\mathcal{T}_n$ with the vertex set of the rooted $n$-regular tree with the root $v_0\in\mathcal{T}_n$.

%\begin{remark}
%    \end{remark}

%\begin{remark}
%\end{remark}
%\begin{remark}
%    For $n=\infty$,
%    $\overline{\mathcal{T}}_\infty$ is naturally identified with the boundary path space $\partial E_\infty$ of the graph $E_\infty:=(E^0_\infty, E^1_\infty, r, s)$ such that
%    \[E^0_\infty=(E^0_\infty)_{\text{sing}}=\{V_0\}, \quad E^1_\infty=\{e_i\}_{i=1}^\infty,\quad r(e_i)=s(e_i)=V_0.\]
%\end{remark}
%By the above remarks, the topological space $\overline{\mathcal{T}}_n$ ($n=2,\cdots, \infty$) is identified with the boundary path space $\partial E_n$ (see \cite[Sec. 8.1.]{NO}).

\subsection{Graph groupoids  for $\mathcal{E}_n$ and $\mathcal{O}_\infty$}
Let $\sigma_{E_n} \colon \partial E_n\backslash \{v_0\}\ni \mu_1\mu_2\mu_3\cdots\;\mapsto \mu_2\mu_3\cdots\;\in\partial E_n$ be the partially defined shift map which is a local homeomorphism.
Following \cite[Sec. 8.3.]{NO},
we introduce the graph groupoid $\mathcal{G}_{E_n}$ as the following Deaconu--Renault groupoid
\[\mathcal{G}_{E_n}:=\{(x, k-l, y)\in \partial E_n\times \mathbb{Z}\times\partial E_n\;|\; k, l\in\mathbb{Z}_{\geq 0},\;\; \sigma_{E_n}^k(x)=\sigma_{E_n}^l(y)\}\]
with the unit space \[\mathcal{G}_{E_n}^{(0)}:=\{(x, 0, x)\;|\; x\in\partial E_n\}=\partial E_n\]
and the structure maps
\[s(x, m, y):=y,\quad r(x, m, y):=x, \quad (x, m, y)\cdot (y, n, z):=(x, m+n, z).\]
The topology of $\mathcal{G}_{E_n}$ is given by the following open basis
\[Z(U, k, l, V):=\{(x, k-l, y)\;|\; x\in U, \;y\in V,\; \sigma_{E_n}^k(x)=\sigma_{E_n}^l(y)\}, \quad U, V\subset \partial E_n \colon \text{clopen},\]
and $\mathcal{G}_{E_n}$ becomes a locally compact, second countable, ample groupoid.

By \cite[Prop. 8.2.]{NO},
$\mathcal{G}_{E_n}$ is topologically principal (i.e., the subset of points with trivial isotropy in $\mathcal{G}_{E_n}^{(0)}$  is dense) for $n=2, \cdots, \infty$.
By Remark \ref{r1}, the minimality holds only for $\mathcal{G}_{E_\infty}$.

By \cite[Prop.2.2]{BCW},
one has the following isomorphisms
\[C_r^*(\mathcal{G}_{E_n})\ni 1_{Z(Z(i), 1, 0, Z(v_0))}\mapsto S_{e_i}+S_{f_i}=T_i\in C^*(E_n)=\mathcal{E}_n, \quad (2 \leq n<\infty),\]
\[C_r^*(\mathcal{G}_{E_\infty})\ni 1_{Z(Z(i), 1, 0, Z(v_0))}\mapsto S_{e_i}=T_i\in C^*(E_\infty)=\mathcal{O}_\infty.\]
The Cartan subalgebra $C(\partial E_n)$  is identified with \[\overline{\operatorname{span}}\{1, T_\mu T_\mu^*\;|\; \mu\in \bigcup_{k=1}^\infty\{1, \cdots, n\}^k\}, \;\; 2 \leq n\leq\infty.\]

By the previous arguments, the  groupoid C*-algebras $C^*_r(\mathcal{G}_{O_n}), C^*_r(\mathcal{G}_{E_n}), C^*_r(\mathcal{G}_{E_\infty})$ are canonically isomorphic to $\mathcal{O}_n$, $\mathcal{E}_n$ and $\mathcal{O}_\infty$, respectively.
%\[\mathcal{E}_n\ni T_i\mapsto S_{e_i}+S_{f_i}\in C^*(E_n),\]
%\[\mathcal{O}_\infty\ni T_i\mapsto S_{e_i}\in C^*(E_\infty).\]
Thus, K-theory computes the groupoid homologies of the above groupoids by HK conjecture.
\begin{thm}[{\cite[Thm. 4.6.]{NO2}}]
For $n=2, \cdots, \infty$,
we have
   \[H_0(\mathcal{G}_{E_n})=\mathbb{Z},\quad H_k(\mathcal{G}_{E_n})=0, \;\; k\geq 1.\]
For $n<\infty$,
we have
   \[H_0(\mathcal{G}_{O_n})=\mathbb{Z}/(n-1)\mathbb{Z},\quad H_k(\mathcal{G}_{O_n})=0, \;\; k\geq 1.\]
\end{thm}

%%%%%%%%%%%%%%%%%%%%%%%%%%%%%%%%%%%%%%%%%%%%%%%%%%%%%%%%%%%%%%%%%%%%%%%%%%%%%%%%%%%%%%%%%%%%%%%%%%%%%%%%%%%%%%%%%%%%%%%%%%%%%%%%%%%%%%%%%%%%%%%%%%%%%%%%%%%%%%%%%%%%%%%%%%%%%%%%%%%%%%%%%%%%%%%%%%%%%%%%%%%%%%%%%%%%%%%%%%%%%%%%%%%%%%%%%%%%%%%%%%%%%%%%%%%%%%%%%%%%%%%%%%%%%%%%%%%%%%%%%%%%%%%%%%%%%%%%%%%%%%%%%%%%%%%%%%%%%%%%%%%%%%%%%%%%%%

\subsection{Homology group $H_0(\mathcal{G})$}
We will use X. Li's result \cite{XinLi}, and we recall the 0th homology groups. 
The reader may refer to \cite{XinLi}, \cite[Sec. 3.1.]{matui} for more details.
For a locally compact, ample groupoid $\mathcal{G}$,
the set of $\mathbb{Z}$-valued continuous, compactly supported function $C_c(\mathcal{G}, \mathbb{Z})$ consists of the elements
\[\sum_{i=1}^M a_i 1_{U_i},\;\; a_i\in\mathbb{Z},\;\;\; U_i\subset \mathcal{G}\;: \;\text{mutually disjoint clopen bisection}.\]
There is a well-defined map
\[\partial_1 \colon C_c(\mathcal{G}, \mathbb{Z})\ni 1_U\mapsto 1_{s(U)}- 1_{r(U)}\in C_c(\mathcal{G}^{(0)}, \mathbb{Z}),\]
and 0th homology group is defined by
\[H_0(\mathcal{G}):=C_c(\mathcal{G}^{(0)}, \mathbb{Z})/\operatorname{Im}\partial_1.\]
\begin{ex}
    The map
    \[H_0(R_F)=C_c(F, \mathbb{Z})/\operatorname{Im}\partial_1\ni f+\operatorname{Im}\partial_1\mapsto \sum_{x\in F}f(x)\in\mathbb{Z}\]
    is an isomorphism.
\end{ex}
Note that one also has $H_{n\geq 1}(R_F)=0$ by definition.
\begin{ex}
    A generator of $H_0(\mathcal{G}_{O_n})=\mathbb{Z}/(n-1)\mathbb{Z}$ is given by $1_{Z^\infty(1)}+\operatorname{Im}\partial_1$, and one has
    \[1_{Z^\infty(i)}+\operatorname{Im}\partial_1=1_{Z^\infty(1)}+\partial_1(1_{Z(Z^\infty(1), 1, 1, Z^\infty(i))})=1_{Z^\infty(1)}+\operatorname{Im}\partial_1,\]
    \[1_{Z^\infty(1)}+\operatorname{Im}\partial_1=1_{\partial O_n}+\partial_1(1_{Z(Z^\infty(v_0), 0, 1, Z^\infty(1))})=1_{\partial O_n}+\operatorname{Im}\partial_1,\]
    \[(1-n)(1_{Z^\infty(1)}+\operatorname{Im}\partial_1)=(1_{\partial O_n}-\sum_{i=1}^n 1_{Z^\infty(i)})+\operatorname{Im}\partial_1=0\in H_0(\mathcal{G}_{O_n}).\]
\end{ex}
\begin{ex}\label{1-n}
    A generator of $1\in H_0(\mathcal{G}_{E_n})=\mathbb{Z}$ is given by $1_{\partial E_n}+\operatorname{Im}\partial_1=1_{Z(1)}+\operatorname{Im}\partial_1$.
    For $2 \leq n < \infty$,
    one has \[1_{\{v_0\}}+\operatorname{Im}\partial_1=(1-n)\in H_0(\mathcal{G}_{E_n})=\mathbb{Z}\] by the computation
    \begin{align*}
        1_{\{v_0\}}=&1_{\partial E_n}-\sum_{i=1}^n 1_{Z(i)}\\
        =&1_{Z(1)}+\partial_1(1_{Z(Z(1), 1, 0, Z(v_0))})-n1_{Z(1)}-\sum_i\partial_1(1_{Z(Z(1), 1, 1, Z(i))})\\
        \in&(1-n)1_{Z(1)}+\operatorname{Im}\partial_1.
    \end{align*}
\end{ex}

%%%%%%%%%%%%%%%%%%%%%%%%%%%%%%%%%%%%%%%%%%%%%%%%%%%%%%%%%%%%%%%%%%%%%%%%%%%%%%%%%%%%%%%%%%%%%%%%%%%%%%%%%%%%%%%%%%%%%%%%%%%%%%%%%%%%%%%%%%%%%%%%%%%%%%%%%%%%%%%%%%%%%%%%%%%%%%%%%%%%%%%%%%%%%%%%%%%%%%%%%%%%%%%%%%%%%%%%%%%%%%%%%%%%%%%%%%%%%%%%%%%%%%%%%%%%%%%%%%%%%%%%%%%%%%%%%%%%%%%%%%%%%%%%%%%%%%%%%%%%%%%%%%%%%%%%%%%%%%%%%%%%%%%%%%%%%%%%%%%%%%%%%%%%%%%%%%%%%%%%%%%%%%%%%%%%%%%%%%%%%%%%%%%%%%%%%%%%%%%%%%%%%%%%%%%%%%%%%%%%%%%%%%%%%%%%%%%%%%%%%%%%%%%%%%%%%%%%%%%%%%%%%%%%%%%%%%%%%%%%%%%%%%%%%%%%%%%%%%%%%

\subsection{Higman--Thompson's groups $[[\mathcal{G}_{O_n}]]$}\label{htg}
We recall V. Nekrashevych's picture of Higman--Thompson groups $V_n=[[\mathcal{G}_{O_n}]]$ (see \cite{Nek}). % and identify $V_n$ with the topological full group $[[\mathcal{G}_{O_n}]]$.
Since $\mathcal{G}_{O_n}$ is topologically principal, locally compact, ample groupoid,
we identify $[[\mathcal{G}_{O_n}]]$ with the subgroup of $\operatorname{Homeo}(\partial O_n)$. %consisting $\pi_U,\; U\in[[\mathcal{G}_{O_n}]]$.
A full bisection $U$ of $\mathcal{G}_{O_n}$ is given by a pair of partitions of clopen sets
\[U:=\bigsqcup_{i=1}^MZ(Z^\infty(\mu_i), |\mu_i|, |\nu_i|, Z^\infty(\nu_i)),\]
where $\mu _i$ and $\nu _i$ satisfy
\[\bigsqcup_{i=1}^MZ^\infty(\mu_i)=\partial O_n =\bigsqcup_{i=1}^MZ^\infty(\nu_i),\]
and we write \[U={_{\mu_i}g_{\nu_i}} \colon \partial O_n\ni \nu_i x\mapsto \mu_i x\in \partial O_n\]
for short.
An element of the Higman--Thompson group $V_n$ is naturally identified with the pair of partitions (cf \cite{Hig}),
and one can identify $V_n$ with $[[\mathcal{G}_{O_n}]]$.
%Thus, we write $V_n=[[\mathcal{G}_{O_n}]]$ for short.
%Note that the different partitions
%\[Z(11)\sqcup Z(12)\sqcup Z(21)\sqcup Z(22)\sqcup Z(3)\cdots Z(n)=Z(21)\sqcup Z(22)\sqcup Z(11)\sqcup Z(12)\sqcup Z(3)\cdots Z(n),\]
%\[Z(2)\sqcup Z(1)\sqcup Z(3)\cdots Z(n)=Z(1)\sqcup Z(2)\sqcup Z(3)\cdots Z(n),\]
%\[Z(1)\sqcup Z(2)\sqcup Z(3)\cdots Z(n)=Z(2)\sqcup Z(1)\sqcup Z(3)\cdots Z(n)\]
%may represent the same element in $[[\mathcal{G}_{O_n}]]$.

V. Nekrashevych gives the following picture to understand $V_n$ as a subgroup of $U(\mathcal{O}_n)$.
\begin{prop}[{\cite[Prop. 9.5.]{Nek}, Lem. \ref{matui}}]
    The map 
    \[V_n\ni {_{\mu_i}g_{\nu_i}}\mapsto \sum_{i=1}^MS_{\mu_i}S_{\nu_i}^*\in U(\mathcal{O}_n)\]
    is a well-defined injective group homomorphism by which we identify $V_n$ with the subgroup
    \begin{align*}&\{\sum_{i=1}^M S_{\mu_i}S_{\nu_i}^*\in U(\mathcal{O}_n)\;|\; \sum_{i=1}^MS_{\mu_i}S_{\mu_i}^*
    =\sum_{i=1}^M S_{\nu_i}S_{\nu_i}^*=1_{\mathcal{O}_n},\;\; M\in \mathbb{N}\}\\=&\{1_U\in N(C^*_r(\mathcal{G}_{O_n}), C(\partial O_n))\; |\; U\in [[\mathcal{G}_{O_n}]]\}.\end{align*}
\end{prop}
Thus, we also write
\[{_{\mu_i}g_{\nu_i}}=\sum_i S_{\mu_i}S_{\nu_i}^*\]
for short.
%Applying Lem. \ref{matui},
%The above map $V_n\mapsto U(\mathcal{O}_n)$ is identified the splitting in Lem. \ref{matui}, 3.
One has \[1_U=_{\mu_i}g_{\nu_i}\in C^*_r(\mathcal{G}_{O_n})=\mathcal{O}_n, \quad 1_{Z^\infty(\mu)}=S_\mu S_\mu^*\in C(\partial O_n)\subset\mathcal{O}_n,\]
and
\[{_{\mu_i}g_{\nu_i}} F (_{\mu_i}g_{\nu_i})^*(x)=F({_{\mu_i}g_{\nu_i}}^{-1} (x)), \quad x\in \partial O_n, \;\; F\in C(\partial O_n), \;{_{\mu_i}g_{\nu_i}}\in U(\mathcal{O}_n).\]
%Recall that $1_{Z(\mu_i)}\in C(\partial O_n)$ is identified with $S_{\mu_i}S_{\mu_i}^*\in\mathcal{O}_n$.
%Since \[{_{\mu_i}\pi_{\nu_i}1_{Z(\nu_i)}(_{\mu_i}\pi_{\nu_i})^*=(\sum_i S_{\mu_i}S_{\nu_i}^*)S_{\nu_i}S_{\nu_i}^*(\sum_i S_{\nu_i}S_{\mu_i}^*)=1_{Z(\mu_i)}} \in\mathcal{O}_n,\]
%it is straightforward to check
%\[{_{\mu_i}\pi_{\nu_i}} F (_{\mu_i}\pi_{\nu_i})^*(x)=F({_{\mu_i}\pi_{\nu_i}}^{-1}\cdot x), \quad x\in \partial O_n, \;\; F\in C(\partial O_n), \;{_{\mu_i}\pi_{\nu_i}}\in U(\mathcal{O}_n)\]
%(i.e., the action $V_n\curvearrowright \partial O_n$ is identified with the adjoint action $\operatorname{Ad}{(_{\mu_i}g_{\nu_i})}$).
%\begin{remark}
%    Recall the normalizer group \[N(\mathcal{O}_n, C(\partial O_n)):=\{u\in U(\mathcal{O}_n)\;|\; uC(\partial O_n) u^*=C(\partial O_n)\}\]
%    and one has $V_n\subset N(\mathcal{O}_n, C(\partial O_n))$.
%    This is understood by H. Matui's result (see Lemma \ref{matui}),
%    and $V_n\subset N(\mathcal{O}_n, C(\partial O_n))$ is the image of the splitting map
%    \[[[\mathcal{G}_{O_n}]]\ni U={_{\mu_i}\pi_{\nu_i}}\mapsto 1_U=\sum_i S_{\mu_i}S^*_{\nu_i}\in C^*_r(\mathcal{G}_{O_n})=\mathcal{O}_n.\]
%\end{remark}

%%%%%%%%%%%%%%%%%%%%%%%%%%%%%%%%%%%%%%%%%%%%%%%%%%%%%%%%%%%%%%%%%%%%%%%%%%%%%%%%%%%%%%%%%%%%%%%%%%%%%%%%%%%%%%%%%%%%%%%%%%%%%%%%%%%%%%%%%%%%%%%%%%%%%%%%%%%%%%%%%%%%%%%%%%%%%%%%%%%%%%%%%%%%%%%%%%%%%%%%%%%%%%%%%%%%%%%%%%%%%%%%%%%%%%%%%%%%%%%%%%%%%%%%%%%%%%%%%%%%%%%%%%%%%%%%%%%%%%%%%%%%%%%%%%%%%%%%%%%%%%%%%%%%%%%%%%%%%%%%%%%%%%%%%%%%%%

\subsection{Topological full groups $[[\mathcal{G}_{E_n}]]$}
An analogue of V. Nekrashevych's picture for $[[\mathcal{G}_{E_n}]]$ is given below as a consequence of Lemma \ref{matui}.
%We write $\Gamma_n :=[[\mathcal{G}_{E_n}]]$ for short.
Recall that $e_n :=1_{\mathcal{E}_n}-\sum_{i=1}^nT_iT_i^*$.
\begin{lem}\label{tabl}
    For $2 \leq n<\infty$, the following set
    \begin{equation*}
\Gamma_n:=\left\{ \sum _{i = 1}^{M} T_{\mu _i} T_{\nu _i}^* + \sum _{k=1}^N T_{v_k} e_n T_{w_k}^*\in \mathcal{E}_n \middle|
\begin{gathered}
\mu _i, \nu _i, v_k, w_k \in E^f_n = \{v_0\}\cup \bigcup_{k=1}^\infty \{1, \cdots, n\}^k ,\\
\bigsqcup_{i=1}^MZ(\mu_i)\sqcup \bigsqcup_{k=1}^N\{v_k\}=\partial E_n=\bigsqcup_{i=1}^MZ(\nu_i)\sqcup\bigsqcup_{k=1}^N\{w_k\}
\end{gathered}
\right\}.
\end{equation*}
    \begin{comment}
    \begin{equation*}
\Gamma_n:=\left\{ \sum _{i = 1}^{M} T_{\mu _i} T_{\nu _i}^* + \sum _{k=1}^N T_{v_k} e_n T_{w_k}^*\in \mathcal{E}_n \middle|
\begin{gathered}
\mu _i, \nu _i, v_k, w_k \in \{v_0\}\cup \bigcup_{k=1}^\infty \{1, \cdots, n\}^k=E^f_n ,\\
\bigsqcup _{i=1}^{M}  Z^\infty(\mu _i) = \bigsqcup _{i=1}^{M}  Z^\infty(\nu _i) = E^\infty_n,\;\; T_{v_0}:=1_{\mathcal{E}_n}\\ 
\{ v_k \} _{k=1}^N \text{ are finite words which do not start with any } \mu _i ,\\
 \{ w_k \} _{k=1}^N \text{ are finite words which do not start with any } \nu _i
\end{gathered}
\right\}.
\end{equation*}
\end{comment}
is identified with the image of splitting group homomorphism
\[[[\mathcal{G}_{E_n}]]\ni U\mapsto 1_U\in N(C^*_r(\mathcal{G}_{E_n}), C(\partial E_n)).\]
In particular,
we have $\Gamma_n=[[\mathcal{G}_{E_n}]]$.
\end{lem}
\begin{proof}
    Recall that the open basis of $\partial E_n$ ($n<\infty$) are given by the sets
    \[Z(\mu),\;\; \{\mu\}, \;\;\mu\in E^f_n.\]
    %The conditions for $\mu_i, \nu_i, v_{i_k}, v_{j_k}$ just imply
    %\[\bigsqcup_{i=1}^MZ(\mu_i)\sqcup \bigsqcup_{k=1}^N\{v_k\}=\partial E_n=\bigsqcup_{i=1}^MZ(\nu_i)\sqcup\bigsqcup_{k=1}^N\{w_k\}.\]
    Thus, the following set is a full bisection
    \[U:=\bigsqcup_{i=1}^MZ(Z(\mu_i), |\mu_i|, |\nu_i|, Z(\nu_i))\sqcup \bigsqcup_{k=1}^NZ(\{v_k\}, |v_k|, |w_k|, \{w_k\})\in [[\mathcal{G}_{E_n}]]\]
    and one has 
    \[1_U=\sum _{i = 1}^{M} T_{\mu _i} T_{\nu _i}^* + \sum _{k=1}^N T_{v_k} e_n T_{w_k}^*\in C^*_r(\mathcal{G}_{E_n})=\mathcal{E}_n.\]
    By the compactness of $\partial E_n$ and $E^\infty_n$,
    every full bisection is of the above form.
    Thus, the image of splitting $[[\mathcal{G}_{E_n}]]\ni U\mapsto 1_U\in N(\mathcal{E}_n, C(\partial E_n))$ is equal to $\Gamma_n$.
\end{proof}
Since $e_n=e_{n+1}+T_{n+1}T_{n+1}^*\in \mathcal{E}_n\subset\mathcal{E}_{n+1}$, we have the canonical inclusion $\Gamma_n<\Gamma_{n+1}$ in the algebras $\mathcal{E}_n\subset \mathcal{E}_{n+1}\subset \mathcal{O}_\infty$.
We define $\Gamma_\infty:=\bigcup_{n=2}^\infty \Gamma_n\subset \mathcal{O}_\infty$.
For the topologically principal, second countable, locally compact,  ample groupoid $\mathcal{G}_{E_\infty}$,
    Lemma \ref{matui} gives the exact sequence
    \[1\to U(C(\partial E_\infty))\to N(\mathcal{O}_\infty, C(\partial E_\infty))\to [[\mathcal{G}_{E_\infty}]]\to 1.\]
\begin{lem}
    The group $\Gamma_\infty$ is canonically identified with $[[\mathcal{G}_{E_\infty}]]$ via the splitting map \[[[\mathcal{G}_{E_\infty}]]\ni U\mapsto 1_U\in N(\mathcal{O}_\infty, C(\partial E_\infty))\]
    in Lem. \ref{matui}.
\end{lem}
\begin{proof}
%    For topologically principal, etale, second countable groupoids $[[\mathcal{G}_{E_n}]]$,
%    Lemma \ref{matui} gives the following commutative diagram with exact horizontal sequences:
%    \[\xymatrix{
%    U(C(\partial E_\infty))\ar[r]&N(\mathcal{O}_\infty, C(\partial E_\infty))\ar[r]&[[\mathcal{G}_{E_\infty}]]\\
%    U(C(\partial E_n))\ar[u]\ar[r]&N(\mathcal{E}_n,
%    C(\partial E_n))\ar[u]^{i_n}\ar[r]&[[\mathcal{G}_{E_n}]]\ar@{-->}[u]\\
%    &\Gamma_n,\ar@{=}[ur]\ar[u]&
%    }\]
%    where the map $i_n$ is well-defined because $U(C(\partial E_n)), \;U(C(\partial E_\infty))$ are commutative and $N(\mathcal{E}_n, C(\partial E_n))=\Gamma_n\cdot U(C(\partial E_n))$ normalize $C(\partial E_\infty)$ (i.e., $g1_{Z(\mu)}g^*=gT_\mu T_\mu^* g^*\in C(\partial E_\infty)=\overline{\operatorname{span}}\{1, T_\mu T_\mu^*\}$ for $g\in \Gamma_n$).

 %   Now $\Gamma_n\subset N(\mathcal{O}_\infty, C(\partial E_\infty))\to [[\mathcal{G}_{E_\infty}]]$ gives a map $f : \Gamma_\infty\to [[\mathcal{G}_{E_\infty}]]$.
%If $g\in \operatorname{Ker} f$,
%one has $g\in \Gamma_n \cap U(C(\partial E_\infty))$ for some $n\in\mathbb{N}$.
%Since $C(\partial E_n)\subset C(\partial E_\infty)$ are commutative and $F(x)=gFg^*(x)=F(g^{-1}\cdot x)$ for every $F\in C(\partial E_n), \; x\in\partial E_n$,
%one has $g=e\in \Gamma_n$.
%Thus, the map $f$ is injective.
%
For $\sum_{i=1}^MT_{\mu_i}T_{\nu_i}^*+\sum_{k=1}^NT_{v_k}e_n T_{w_k}^*\in \Gamma_n$,
one has
\begin{align*}
    \sum_{i=1}^MT_{\mu_i}T_{\nu_i}^*+\sum_{k=1}^NT_{v_k}e_n T_{w_k}^*=&\sum_{i=1}^MT_{\mu_i}T_{\nu_i}^*+\sum_{k=1}^N(T_{v_k}T_{w_k}^*-\sum_{j\in \{1,\cdots, n\}}T_{v_kj} T_{w_kj}^*)\\
    =&1_{\bigsqcup_{i=1}^MZ(Z(\mu_i), |\mu_i|, |\nu_i|, Z(\nu_i))\sqcup \bigsqcup_{k=1}^NZ(Z(v_k)\backslash(\sqcup_{j=1}^nZ(v_k j)), |v_k|, |w_k|, Z(w_k)\backslash(\sqcup_{j=1}^nZ(w_k j)))}\\
    =:&1_V.
\end{align*}
Since $\sum_{i=1}^MT_{\mu_i}T_{\nu_i}^*+\sum_{k=1}^NT_{v_k}e_n T_{w_k}^*=1_V\in U(\mathcal{E}_n)\subset U(\mathcal{O}_\infty)$ is a unitary (i.e., $1_{s(V)}=(1_V)^*1_{V}=1_{\partial E_\infty}=1_V(1_V)^*=1_{r(V)}$),
one has
\[\bigsqcup_{i=1}^M Z(\mu_i)\sqcup\bigsqcup_{k=1}^N(Z(v_k)\backslash (\sqcup_{j=1}^nZ(v_k j)))=\partial E_\infty=\bigsqcup_{i=1}^M Z(\nu_i)\sqcup\bigsqcup_{k=1}^N(Z(w_k)\backslash (\sqcup_{j=1}^nZ(w_k j)))\]
which implies $V\in [[\mathcal{G}_{E_\infty}]]$ and that every $\Gamma_n$ is contained in the image of the splitting map.

By \cite[Prop. 9.4.]{NO}, an arbitrary element $U\in [[\mathcal{G}_{E_\infty}]]$ is represented by
    \[U:=\bigsqcup_{i\in I}Z(Z(\mu_i)\backslash(\sqcup_{k\in F_i}Z(\mu_i k)), |\mu_i|, |\nu_i|, Z(\nu_i)\backslash(\sqcup_{\nu_i k\in F_i}Z(k))),\]
    with $\mu_i, \nu_i\in E^f_\infty$ and a finite subset $F_i\subset \mathbb{N}$ satisfying
    \[\bigsqcup_{i\in I}(Z(\mu_i)\backslash(\sqcup_{k\in F_i}Z(\mu_i k)))=\partial E_\infty=\bigsqcup_{i\in I}(Z(\nu_i)\backslash(\sqcup_{k\in F_i}Z(\nu_i k))).\]
%A lift of $U\in [[\mathcal{G}_{E_\infty}]]$ in $N(\mathcal{O}_\infty, C(\partial E_\infty))$ is given by
One has
\[1_U=\sum_{i\in I} T_{\mu_i}(1-\sum_{k\in F_i}T_kT_k^*)T^*_{\nu_i}\in N(\mathcal{O}_\infty, C(\partial E_\infty)).\]
%by \cite[Prop. 5.6.]{matui}.
Since $F_i, I$ are all finite,
there is $N\in\mathbb{N}$ such that $\{\mu_i, \nu_i,  \mu_i k, \nu_i k\}_{i\in I, k\in F_i}\subset \partial E^f_N$ (i.e., $1_U\in U(\mathcal{E}_n)\subset U(\mathcal{O}_\infty)$).
The direct compuation yields
\begin{align*}
    1_U=&\sum_{i\in I} T_{\mu_i}(1-\sum_{k\in F_i}T_kT_k^*)T^*_{\nu_i}\\
    =&\sum_{i\in I} T_{\mu_i}(1-\sum_{k=1}^NT_kT_k^*+\sum_{l\in \{1,\cdots, N\}\backslash F_i}T_l T_l^*)T^*_{\nu_i}\\
    =&\left(\sum_{i\in I}\sum_{l\in \{1, \cdots, N\}\backslash F_i}T_{\mu_i l}T_{\nu_i l}^*\right)+\sum_{i\in I}T_{\mu_i}e_n T^*_{\nu_i},
\end{align*}
and $1_U\in U(\mathcal{E}_n)$ implies
\[\bigsqcup_{i\in I, \; l\in \{1, \cdots, N\}\backslash F_i}Z(\mu_i l)\sqcup \bigsqcup_{i\in I}\{\mu_i\}=\partial E_n=\bigsqcup_{i\in I, \; l\in \{1, \cdots, N\}\backslash F_i}Z(\nu_i l)\sqcup \bigsqcup_{i\in I}\{\nu_i\}\]
(i.e., $1_U\in \Gamma_N$).
So one can conclude $\{1_U\in N(\mathcal{O}_\infty, C(\partial E_\infty))\;|\; U\in [[\mathcal{G}_{E_\infty}]]\}=\bigcup_{n=2}^\infty\Gamma_n=\Gamma_\infty$.
%So one has $1_U\in U(\mathcal{E}_N)$ and it is easy to check that $1_U T_\xi T_\xi^* (1_U)^*\in \overline{\operatorname{span}}\{T_\nu T_\nu^* \in \mathcal{E}_N\}$ for $\xi\in\partial E^f_N$.
%Thus, one has 
%\[1_U\in N(\mathcal{E}_N)=\Gamma_N \cdot U(C(\partial E_n))\subset \Gamma_\infty \cdot U(C(\partial E_\infty))\] which implies $g\in f(\Gamma_\infty)$.
\end{proof}
%Note that in the above proof the sets $\{\mu_i\}_{i\in I}$ and $\{\nu_i\}_{i\in I}$ must contain $v_0$ and we use the convention $T_{v_0}=1$.
\begin{lem}\label{lift}
    \begin{enumerate}
    \item For any presentation $g=\sum_{i=1}^M S_{\mu_i}S_{\nu_i}^*\in V_n$ of the element $g\in V_n$,
        there exist $\{ v_k \}_k , \{ w_k \} _k \subset E_n^f$ such that
        \[\sum_{i=1}^M T_{\mu_i} T_{\nu_i}^*+\sum_kT_{v_k}e_n T_{w_k}^*\in \Gamma_n.\]
        \item For every $\mu, \nu\in E^f_n$ with $|\mu|, |\nu|\geq 1$,
        there are elements of the form
        \[S_\mu S_\nu^*+\sum_i S_{\mu_i}S_{\nu_i}^*\in V_n,\quad (T_\mu T_\nu^*+\sum_i T_{\mu_i}T_{\nu_i}^*)+\sum_kT_{v_k}e_n T_{w_k}^*\in \Gamma_n.\]
        \item Fix $\mu\in E^f_n$.
        For an element $g\in V_n$ satisfying $g(x)=x$, for every $x\in Z^\infty(\mu)$,
        one has a presentation
        \[g=S_\mu S_\mu^*+\sum_iS_{\mu_i}S_{\nu_i}^*.\]
    \end{enumerate}
\end{lem}
\begin{proof}
\begin{enumerate}
\item Note that
\[\pi (1-\sum_iT_{\mu_i}T_{\mu_i}^*)=1-gg^{-1}=0\]
implies taht $1-\sum_{i=1}^MT_{\mu_i}T_{\mu_i}^*$ is a finite rank projection.
Since $\sum_{i=1}^MS_{\mu_i}S_{\nu_i}^*\in U(\mathcal{O}_n)$ and $K_1(\mathcal{O}_n)=0$,
the Fredholm index computation yields
\begin{align*}|\{v\in E^f_n \mid v \;\text{does not start with any}\; \mu_i\}|=&\sum_{v\in E^f_n} \langle \left(1-\sum_{i=1}^M T_{\mu_i}T_{\mu_i}^*\right)\delta_v |\delta_v \rangle_{\ell^2(E^f_n)}\\
=&\operatorname{dim}_{\mathbb{C}}\operatorname{Im}(1-\sum_{i=1}^MT_{\mu_i}T_{\mu_i}^*)\\
=&\operatorname{dim}_{\mathbb{C}}\operatorname{Im}(1-\sum_{i=1}^MT_{\nu_i}T_{\nu_i}^*)\\
=&|\{w\in E^f_n \mid w\;\text{does not start with any}\; \nu_i\}|=:N<\infty.\end{align*}
Thus, there are $\{v_k\}_{k=1}^N,\;\;\{w_k\}_{k=1}^N\subset E^f_n$ and a well-defined lift $\sum_{i=1}^MT_{\mu_i}T_{\nu_i}^*+\sum_{k=1}^NT_{v_k}e_n T_{w_k}^*\in\Gamma_n$.

\item By the statement 1., it is enough to show the case of $V_n$.
If $Z^\infty(\mu)\cup Z^\infty(\nu)=\partial O_n=E^\infty_n$,
one has $S_\mu S_\nu^*+S_\nu S_\mu^*\in V_n$.
%Without loss of generality,
%we may assume $\mu=\overbrace{1\cdots 1}^{|\mu|},\;\nu=\overbrace{1\cdots 1}^{|\nu|}$.
Otherwise, there is a decomposition $(Z^\infty(\mu)\cup Z^\infty(\nu))^c=\bigsqcup_i Z^\infty(\mu_i)$ and an element
$S_\mu S_\nu^*+S_\nu S_\mu^*+\sum_i S_{\mu_i}S_{\mu_i}^*\in V_n$.

\item For a presentation $g=\sum_iS_{\mu'_i}S_{\nu'_i}^*\in V_n$,
the subdivision $g=\sum_i\sum_{\nu, \;|\nu|=|\mu|}S_{\mu'_i\nu}S_{\nu'_i\nu}^*$ gives us the dichotomy $Z^\infty(\mu)\cap Z^\infty(\mu'_i\nu)=Z^\infty(\mu'_i\nu)$, or $Z^\infty(\mu)\cap Z^\infty(\mu'_i\nu)= \emptyset$.
Thus we may assume \[g=\sum_j S_{\zeta_j}S_{\eta_j}^*+\sum_i S_{\mu_i}S_{\nu_i}^*,\quad \bigsqcup_j Z^\infty(\zeta_j)=Z^\infty(\mu)=\bigsqcup_j Z^\infty(\eta_j).\]
The assumption $g(x)=x$ implies $\zeta_j=\eta_j$.
Thus, we can conclude
\[g=\sum_j S_{\zeta_j}S_{\zeta_j}^*+\sum_i S_{\mu_i} S_{\nu_i}^*=S_\mu S_\mu^*+\sum_i S_{\mu_i} S_{\nu_i}^*.\]
\end{enumerate}
\end{proof}

\subsection{Abelianizations of $V_n, \Gamma_\infty$}
We review the abelianizations of the topological full groups $V_n, \Gamma_\infty$ from the viewpoint of AH conjecture.
By X. Li's recent breakthrough,
we can check the AH conjecture for various ample groupoids.
Note that groupoids $\mathcal{G}_{O_n}, \mathcal{G}_{E_\infty}$ are purely infinite and have comparison (i.e., for any non-empty clopen sets $U, V \subset \mathcal{G}^{(0)}$ there is a bisection $\tau$ with $s(\tau)=U, r(\tau)\subset V$).
For $2 \leq n<\infty$, the groupoid $\mathcal{G}_{E_n}$ is not minimal and does not have comparison (consider $U=\partial E_n$ and $V=\{v_0\}$).
\begin{remark}
Roughly speaking, the minimality and comparison property of groupoids correspond to the simplicity and purely infiniteness of groupoid C*-algebras.
    Thus, the above observations are obvious in the operator algebraic sense because $\mathcal{O}_n, \mathcal{O}_\infty$ are simple, purely infinite but $\mathcal{E}_n$ are not simple nor purely infinite.
\end{remark}
Combining the homology computations for $\mathcal{G}_{E_\infty}$ and $\mathcal{G}_{O_n}$ with X. Li's result and H. Matui's stability result \cite[Thm. 3.6.]{matui},
we obtain the following.
For a group $G$, we write the commutator subgroup as $G'$ and the abelianization as $G^{ab}:=G/G'$,
and we write the quotient map as $G\ni g\mapsto [g]^{ab}\in G^{ab}$.
\begin{thm}[{\cite[Thm. 6.12., Cor. 6.14.]{XinLi}}]\label{AH}
    \begin{enumerate}
        \item We have an isomorphism 
        \[H_0(\mathcal{G}_{O_n})\otimes\mathbb{Z}/2\mathbb{Z}\ni (1_{Z(1)}+\operatorname{Im}\partial_1)\otimes\bar{1} \mapsto [g_0]^{ab}\in V_n^{ab}\]
        %sending $1_{Z(1)}+\operatorname{Im}\partial_1=\bar{1}\in H_0(\mathcal{G}_{O_n})=\mathbb{Z}/(n-1)\mathbb{Z}$ to 
        where
        \[g_0:=S_1S_2^*+S_2S_1^*+\sum_{i=3}^nS_iS_i^*\in V_n.\]
        \item The following diagram is commutative
        \[\xymatrix{
        H_0(\mathcal{G}_{E_\infty})\otimes \mathbb{Z}/2\mathbb{Z}\ar[r]^-{\zeta}\ar[d]^{\cong}&\Gamma_\infty^{ab}\ar[d]^{\cong}\\
        H_0(R_\mathbb{N}\times\mathcal{G}_{E_\infty})\otimes\mathbb{Z}/2\mathbb{Z}\ar[r]^-{\zeta^s}&[[R_\mathbb{N}\times\mathcal{G}_{E_\infty}]]^{ab},
        }\]
        where the map $\zeta$ is an isomorphism sending $1_{Z(1)}+\operatorname{Im}\partial_1=1\in H_0(\mathcal{G}_{E_\infty})=\mathbb{Z}$ to
        $[1_{U_0}]^{ab}$ for
        \begin{align*}1_{U_0}:=&(1-T_1T_1^*-T_2T_2^*)+T_1T_2^*+T_2T_1^*
        %=&Z(Z(v_0)\backslash (Z(1)\sqcup Z(2)), 0, 0, Z(v_0)\backslash (Z(1)\sqcup Z(2)))\\
        %&\sqcup Z(Z(1), 1, 1, Z(2))\sqcup Z(Z(2), 1, 1, Z(1))=:U_0
        \in \Gamma_2\subset\Gamma_\infty,\end{align*}
        and the isomorphism $\zeta^s$ sends $1_{\{1\}\times Z(1)}+\operatorname{Im}\partial_1=1\in H_0(R_\mathbb{N}\times\mathcal{G}_{E_\infty})=\mathbb{Z}$ to 
        \[\left[((1, 1)\times U_0)\sqcup ((R_\mathbb{N}^{(0)}\backslash(1, 1))\times \mathcal{G}_{E_\infty}^{(0)})\right]^{ab}\in [[R_\mathbb{N}\times\mathcal{G}_{E_\infty}]]^{ab}.\]
        \item For the stabilization $R_\mathbb{N}\times \mathcal{G}_{E_n}$, ($2 \leq n<\infty$),
        we have the isomorphism
        \[H_0(\mathcal{G}_{E_n})\otimes \mathbb{Z}/2\mathbb{Z}\cong H_0(R_\mathbb{N}\times \mathcal{G}_{E_n})\otimes\mathbb{Z}/2\mathbb{Z}\xrightarrow{\zeta^s} [[R_\mathbb{N}\times \mathcal{G}_{E_n}]]^{ab},\]
        which sends $1_{Z(1)}+\operatorname{Im}\partial_1=1\in H_0(\mathcal{G}_{E_n})=\mathbb{Z}$ to
        \[\left[((1, 1)\times U_0)\sqcup ((R_\mathbb{N}^{(0)}\backslash(1, 1))\times \mathcal{G}_{E_n}^{(0)})\right]^{ab}\in [[R_\mathbb{N}\times\mathcal{G}_{E_n}]]^{ab}.\]
    \end{enumerate}
\end{thm}
\begin{proof}
    We check the statement 3. because other statements follow from the same argument.
    The isomorphism $H_0(\mathcal{G}_{E_n})\cong H_0(R_\mathbb{N}\times\mathcal{G}_{E_n})$ sends $1_{\partial E_n}+\operatorname{Im}\partial_1$ to 
    \[1_{\{1\}\times \partial E_n}+\operatorname{Im}\partial_1\in C_c(\mathbb{N}\times \partial E_n, \mathbb{Z})/\operatorname{Im}\partial_1.\]
    Then, the map $\zeta^s$ sends this element to the class of the bisection
    \[((1, 2)\times \partial E_n)\sqcup((2, 1)\times \partial E_n)\sqcup\bigsqcup_{k=3}^N((k, k)\times \partial E_n)\in [[R_{\{1, \cdots, N\}}\times\mathcal{G}_{E_n}]]\subset [[R_\mathbb{N}\times\mathcal{G}_{E_n}]].\]
    Now the completely same computation as
    \[\left(\begin{array}{ccccc}
    T_2&T_1&e_2\\
    0&0&T_1^*\\
    0&0&T_2^*
    \end{array}\right)
    \left(\begin{array}{ccccc}
    0&1&0\\
    1&0&0\\
    0&0&1
    \end{array}\right)
    \left(\begin{array}{ccccc}
    T_2^*&0&0\\
    T_1^*&0&0\\
    e_2&T_1&T_2
    \end{array}\right)
    =\left(\begin{array}{ccccc}
    T_1T_2^*+T_2T_1^*+e_2&0&0\\
    0&1&0\\
    0&0&1
    \end{array}\right)\in U(\mathbb{M}_3(\mathcal{E}_2))
    \]
    shows
    \[\left[((1, 2)\times \partial E_n)\sqcup((2, 1)\times \partial E_n)\sqcup\bigsqcup_{k=3}^N((k, k)\times \partial E_n)\right]^{ab}=\left[((1, 1)\times U_0)\sqcup ((R_\mathbb{N}^{(0)}\backslash (1, 1))\times \mathcal{G}_{E_n}^{(0)})\right]^{ab}.\]
\end{proof}

\begin{remark}
    To the best of the our knowledge,
    there seems to be no previous results on AH conjecture for groupoids such as $\mathcal{G}_{E_n} (2 \leq n<\infty)$, which is not of the form $R_\mathbb{N}\times \mathcal{G}$,  is not minimal, has many isolated points in the unit space $\mathcal{G}_{E_n}^{(0)}$, and the graph $E_n$ has a sink.
    Thus, it would not be so obvious to see $\Gamma_n^{ab}=\mathbb{Z}/2\mathbb{Z}$ which will be observed in Sec.\ref{main1}.
\end{remark}

%%%%%%%%%%%%%%%%%%%%%%%%%%%%%%%%%%%%%%%%%%%%%%%%%%%%%%%%%%%%%%%%%%%%%%%%%%%%%%%%%%%%%%%%%%%%%%%%%%%%%%%%%%%%%%%%%%%%%%%%%%%%%%%%%%%%%%%%%%%%%%%%%%%%%%%%%%%%%%%%%%%%%%%%%%%%%%%%%%%%%%%%%%%%%%%%%%%%%%%%%%%%%%%%%%%%%%%%%%%%%%%%%%%%%%%%%%%%%%%%%%%%%%%%%%%%%%%%%%%%%%%%%%%%%%%%%%%%%%%%%%%%%%%%%%%%%%%%%%%%%%%%%%%%%%%%%%%%%%%%%%%%%%%%%%%%%%

\subsection{$\mathbb{R}$-actions and KMS states}\label{gac}
Let $\mathcal{G}$ be an ample groupoid with a continuous groupoid homomorphism
\[c \colon \mathcal{G}\to\mathbb{Z}, \quad c(gh)=c(g)+c(h), \;\;\; (g, h)\in\mathcal{G}^{(2)}.\]
Then, there is a well-defined $\mathbb{R}$-action
\[\gamma_c(t) \colon C_c(\mathcal{G})\ni f(g)\mapsto e^{ic(g)t}f(g)\in C_c(\mathcal{G})\]
which extends to $\mathbb{R}$-actions on the reduced and full groupoid C*-algebras $C^*_r(\mathcal{G})$ and $C^*(\mathcal{G})$.

For the transformation groupoid
\[\mathcal{G}^{(0)}\rtimes [[\mathcal{G}]]:=\{(U(x), U, x)\in \mathcal{G}^{(0)}\times [[\mathcal{G}]]\times \mathcal{G}^{(0)}\},\]
there is a natural groupoid homomorphism
\[q \colon \mathcal{G}^{(0)}\rtimes[[\mathcal{G}]]\ni (U( x), U, x)\mapsto g_x\in \mathcal{G}\]
where the element $g_x\in\mathcal{G}$ is uniquely determined by $U\cap s^{-1}(x)=\{g_x\}$.
For a bisection $V\subset \mathcal{G}$,
one has
\[q^{-1}(V)\cap \{(U( x), U, x)\in \mathcal{G}^{(0)}\rtimes[[\mathcal{G}]]\}=\{(U( x), U, x)\;|\; x\in s^{-1}(U\cap V)\},\]
and $q$ is continuous.

\begin{dfn}
    We define the following cocycles of the Deaconu--Renault groupoids
    \[c_n \colon \mathcal{G}_{O_n}\ni (x, k, y)\mapsto k\in\mathbb{Z},\]
    \[d_n \colon \mathcal{G}_{E_n}\ni (x, k, y)\mapsto k\in\mathbb{Z}.\]
    The pullbacks $c_n\circ q, d_n\circ q$ are denoted by
    \[c^f_n \colon \partial O_n\rtimes V_n\to \mathcal{G}_{O_n}\xrightarrow{c_n}\mathbb{Z},\]
    \[d^f_n \colon \partial E_n\rtimes \Gamma_n\to\mathcal{G}_{E_n}\xrightarrow{d_n}\mathbb{Z}.\]
\end{dfn}
The above cocycles define the $\mathbb{R}$ actions
\[\gamma_{c_n} \colon \mathbb{R}\curvearrowright\mathcal{O}_n,\quad \gamma_{c^f_n} \colon \mathbb{R}\curvearrowright C(\partial O_n)\rtimes_r V_n,\]
\[\gamma_{d_n} \colon \mathbb{R}\curvearrowright \mathcal{E}_n, \mathcal{O}_\infty,\quad \gamma_{d^f_n} \colon \mathbb{R}\curvearrowright C(\partial E_n)\rtimes_r \Gamma_n.\]

For an $\mathbb{R}$-action $\gamma(t)\in\operatorname{Aut}(A)$ of a C*-algebra,
a $\gamma$-invariant state $\varphi \colon A\to\mathbb{C}$ is called $\gamma \text{-KMS}_\beta$-state for $\beta\in\mathbb{R}_{\geq 0}$ if 
\[\varphi (ab)=\varphi (b\gamma({i\beta})(a)) \]
holds for any $b\in A$ and any analytic element $a\in A$.
Here, an element $a\in A$ is called analytic if the continuous map $\mathbb{R}\ni t\mapsto \gamma(t)(a)\in A$ extends to an entire function $\mathbb{C}\ni z\mapsto \gamma(z)(a)\in A$.

The state $\varphi$ is called ground state (KMS-state for $\beta=+\infty$) if
\[|\varphi(b\gamma(z)(a))|\leq ||b||||a||, \; \text{for}\;\; z\in \mathbb{C},\; \operatorname{Im} z\geq 0\]
holds for any $b\in A$ and any analytic element $a\in A$.
A ground state is automatically $\gamma$-invariant.

We refer to \cite{P} for the basics of the KMS state and ground state.
\begin{remark}
In general,
one can not determine all analytic elements.
However,
it is enough to check the above KMS conditions for a dense subset of analytic elements by \cite[Prop. 8.12.3.]{P}.
    
\end{remark}

\begin{ex}
    The following elements are analytic:
    \[S_\mu S_\nu^*\in\mathcal{O}_n, \quad \gamma_{c_n}(z)(S_\mu S_\nu^*)=e^{iz(|\mu|-|\nu|)}S_\mu S_\nu^*,\]
    \[T_{v_k}e_nT_{w_k}^*\in\mathcal{E}_n,\quad \gamma_{d_n}(z)(T_{v_k}e_nT_{w_k}^*)=e^{iz(|v_k|-|w_k|)}T_{v_k}e_nT_{w_k}^*,\]
    \[1_{Z(\mu_k)}\lambda_{_{\mu_i}g_{\nu_i}}\in C(\partial O_n)\rtimes V_n,\quad \gamma_{c_n}(z)(1_{Z(\mu_k)}\lambda_{_{\mu_i}g_{\nu_i}})=e^{iz(|\mu_k|-|\nu_k|)}1_{Z(\mu_j)}\lambda_{_{\mu_i}g_{\nu_i}},\]
    for $\mu, \nu\in\partial O_n$, $v_k, w_k\in E^f_n$ and $\bigsqcup_{i=1}^NZ^\infty(\mu_i)=\partial O_n=\bigsqcup_{i=1}^NZ^\infty(\nu_i)$, $k\in \{1,\cdots, N\}$.
\end{ex}

For the KMS states of (full) groupoid C*-algebra,
we also refer to S. Neshveyev's general result \cite[Thm. 1.3.]{Nes}.
This result says that
${\gamma_c}\text{-KMS}_\beta$-states on $C^*(\mathcal{G})$ are given by integrating traces on the stabilizers along the quasi-invariant measure of the unit space which provides the cocycle as its Radon--Nikodym derivatives:
\[\varphi \colon C_c(\mathcal{G})\ni f\mapsto \int_{\mathcal{G}^{(0)}}\sum_{s(g)=r(g)=x} f(g)\varphi_x(g)dm(x)\in \mathbb{C}\]
where $m$ is a quasi-invariant measure on $\mathcal{G}^{(0)}$ with its Radon--Nikodym cocycle $e^{-\beta c}$  and $\varphi_x \colon C^*(s^{-1}(x)\cap r^{-1}(x))\to\mathbb{C}, \; x\in \mathcal{G}^{(0)}$ are the traces on the stabilizers satisfying several conditions.

The KMS-states with respect to $\gamma _{c_n} \colon \mathbb{R} \curvearrowright \mathcal{O}_n$, $\gamma _{d_{\infty}} \colon \mathbb{R} \curvearrowright \mathcal{O}_{\infty}$, and $\gamma _{d_n} \colon \mathbb{R} \curvearrowright \mathcal{E}_n$ are computed as follows.

\begin{thm}[cf. \cite{Evans, OP}]
    \begin{enumerate}
        \item There is a $\gamma _{c_n}$-KMS$_{\beta}$-state on $\mathcal{O}_n$ if and only if $\beta = \log n$.
        There is a unique $\gamma _{c_n}$-KMS$_{\log n}$-state $\varphi _n$ on $\mathcal{O}_n$ given by
        \begin{equation*}
            \varphi _n (S_\nu S_\mu ^*) = 
            \begin{cases}
                n^{-| \mu |} & \text{if $\mu = \nu$} \\
                0 & \text{if $\nu \neq \mu$}.
            \end{cases}
        \end{equation*}
        %The only $\gamma _{c_n}\text{-KMS}_\beta$-state on $\mathcal{O}_n$ is \[ \mathcal{O}_n \to \mathcal{O}_n^{\gamma _{c_n}} \cong M_{n^{\infty}} \to \mathbb{C} \] where the first map is the conditional expectation and the second map is the only tracial state on $M_{n^{\infty}}$.
        There is no $\gamma _{c_n}$-ground state on $\mathcal{O}_n$.
        \item There is no $\gamma _{d_{\infty}}$-KMS$_\beta$-state on $\mathcal{O}_{\infty}$ for $\beta \geq 0$.
        There is a unique $\gamma_{d_{\infty}}$-ground state $\varphi _{\infty}$ on $\mathcal{O}_\infty$ given by
        \begin{equation*}
            \varphi _\infty (T_\nu T_\mu ^*) = 
            \begin{cases}
                1 & \text{if $\mu = \nu = \emptyset$} \\
                0 & \text{if $\mu \neq \emptyset$ or $\nu \neq \emptyset$}.
            \end{cases}
        \end{equation*}
        \item There is a $\gamma _{d_n}$-KMS$_\beta$-state on $\mathcal{E}_n$ if and only if $\beta \geq \log n$.
        %For $\beta = \log n$, there is a unique $\gamma _{d_n}$-KMS$_ {\log n}$-state on $\mathcal{E}_n$ given by \[ \mathcal{E}_n \to \mathcal{O}_n \stackrel{\varphi _n}{\to} \mathbb{C}\].
        For $\beta \geq \log n$, there is a unique $\gamma _{d_n}$-KMS$_\beta$-state $\varphi_{n , \beta}$ on $\mathcal{E}_n$ given by 
        \begin{equation*}
            \varphi _{n, \beta} (T_\nu T_\mu ^*) = 
            \begin{cases}
                e^{-| \mu | \beta} & \text{if $\mu = \nu$} \\
                0 & \text{if $\mu \neq \nu$}.
            \end{cases}
        \end{equation*}
        There is a unique $\gamma _{d_n}$-groud state $\varphi _{n, \infty}$ on $\mathcal{E}_n$ given by 
        \begin{equation*}
            \varphi _{n , \infty}(T_\nu T_\mu ^*) = 
            \begin{cases}
                1 & \text{if $\mu = \nu = \emptyset$} \\
                0 & \text{if $\mu \neq \emptyset$ or $\nu \neq \emptyset$}.
            \end{cases}
        \end{equation*}
    \end{enumerate}
\end{thm}

\subsection{The unique trace property}
Let $\Gamma$ be a discrete group.
The reduced group $\mathrm{C}^*$-algebra $\mathrm{C}^*_r (\Gamma) \subset \mathbb{B}(\ell^2(\Gamma))$ has the canonical tracial state $x \mapsto \langle x \delta _e , \delta _e \rangle$.
A group $\Gamma$ is said to have the unique trace property if the canonical tracial state is the only tracial state on $\mathrm{C}^*_r(\Gamma)$.

To state the characterization of the unique trace property, we review boundary actions of groups.
Let $\Gamma \curvearrowright X$ be an action of a discrete group on a compact Hausdorff space.
This action is said to be minimal if there is no non-trivial $\Gamma$-invariant closed subset.
It is said to be strongly proximal if $\Gamma.\mu$ contains some Dirac measure for every $\mu \in \operatorname{Prob} X$.
A compact $\Gamma$-space $X$ is said to be $\Gamma$-boundary if the action is minimal and strongly proximal.
For instance, the canonical action $V_n \curvearrowright \partial O_n$ of the Higman--Thompson group on the Cantor set is the boundary action.
%The amenable radical of a discrete group $\Gamma$ is the largest normal amenable subgroup of $\Gamma$.

\begin{thm}[\cite{BKKO}]
For a discrete group $\Gamma$, the following are equivalent.
\begin{enumerate}
    \item The group $\Gamma$ has the unique trace property.
    \item The group $\Gamma$ admits a faithful boundary.
    \item The only amenable normal subgroup of $\Gamma$ is $\{ e \}$.
\end{enumerate}
\end{thm}

\begin{remark}
    The groups $V_n, \Gamma _{\infty}$ are $\mathrm{C}^*$-simple (see \cite{LBMB}, \cite{SB}), i.e., their reduced group $\mathrm{C}^*$-algebras are simple, while $\Gamma _n$ is not because $\mathfrak{S}_{E_n^f} \lhd \Gamma _n$ is a non-trivial amenable normal subgroup (in fact, this subgroup coincides with the amenable radical of $\Gamma _n$ since $\Gamma _n / \mathfrak{S}_{E_n^f} \cong V_n$ (see Lem. \ref{quotientGamma})).
    By \cite[Thm. 4.1.]{BKKO} and C*-simplicity of $V_n$,
    the traces on $C^*_r(\Gamma_n)$ are in one to one correspondence to the traces on $C^*_r(\mathfrak{S}_{E^f_n})$ which are invariant under the adjoint action of $\Gamma_n$.
    In Cor. \ref{uhoo},
    we will see that every tarce of $C^*_r(\mathfrak{S}_{E^f_n})$ is automatically $\Gamma_n$-invariant.
\end{remark}
\begin{comment}
\begin{remark}\label{soft}
    For $n<\infty$,
    the group $\Gamma_n$ is not C*-simple because $C^*_r(\Gamma_n)$ has at least two different traces
    \[C^*_r(\Gamma_n)\ni \lambda_g\mapsto 1_{\mathfrak{S}_{E^f_n}}(g)\lambda_g\in C^*_r(\mathfrak{S}_{E^f_n})\xrightarrow{\tau_0, \tau_1} \mathbb{C}, \]
    where $\tau_0(\lambda_g):=\delta_{g, e}$ and $\tau_1(\lambda_g):=1$ for $g\in\mathfrak{S}_{E^f_n}$.
    By \cite[Thm. 4.1.]{BKKO} and C*-simplicity of $V_n$,
    the traces on $C^*_r(\Gamma_n)$ are in one to one correspondence to the traces on $C^*_r(\mathfrak{S}_{E^f_n})$ which are invariant under the adjoint action of $\Gamma_n$.
    In Cor. \ref{uhoo},
    we will see that every tarce of $C^*_r(\mathfrak{S}_{E^f_n})$ is automatically $\Gamma_n$-invariant.
\end{remark}
\end{comment}

%%%%%%%%%%%%%%%%%%%%%%%%%%%%%%%%%%%%%%%%%%%%%%%%%%%%%%%%%%%%%%%%%%%%%%%%%%%%%%%%%%%%%%%%%%%%%%%%%%%%%%%%%%%%%%%%%%%%%%%%%%%%%%%%%%%%%%%%%%%%%%%%%%%%%%%%%%%%%%%%%%%%%%%%%%%%%%%%%%%%%%%%%%%%%%%%%%%%%%%%%%%%%%%%%%%%%%%%%%%%%%%%%%%%%%%%%%%%%%%%%%%%%%%%%%%%%%%%%%%%%%%%%%%%%%%%%%%%%%%%%%%%%%%%%%%%%%%%%%%%%%%%%%%%%%%%%%%%%%%%%%%%%%%%%%%%%%%%%%%%%%%%%%%%%%%%%%%%%%%%

\section{Main Results}
First, we will compute the normal subgroups and abelianization of $\Gamma_n$.
Then, we will determine the KMS states of $C(\partial O_n)\rtimes_r V_n$ and $C(\partial E_n)\rtimes_r\Gamma_n$ with respect to the $\mathbb{R}$-actions defined in Sec. \ref{gac}.

\subsection{Abelianization and normal subgroups of $\Gamma_n$}\label{main1}
We will show the following.
\begin{thm}\label{ns}
     A non-trivial normal subgroup of $\Gamma_n$ is either $\Gamma_n'$, $\mathfrak{S}_{E^f_n}$ or $\mathfrak{S}_{E^f_n}'$.
     A non-trivial normal subgroup of $\Gamma_n'$ is either $\mathfrak{S}_{E^f_n}$ or $\mathfrak{S}_{E^f_n}'$.
\end{thm}
By the above observation, we get the previously known result on the simplicity of $\Gamma'_\infty$.
\begin{cor}[{cf \cite{matu2, NO}}]
    The commutator subgroup $\Gamma_\infty'$ is simple.
\end{cor}
\begin{proof}
    Note that $\Gamma_\infty'=\bigcup_{n=2}^\infty\Gamma_n'$.
    For a normal subgroup $\{e\}\subsetneq N\subsetneq\Gamma_\infty'$,
    there is $m_0\in\mathbb{N}$ satisfying 
    \[\{e\}\subsetneq N\cap\Gamma_m'\subsetneq \Gamma_m',\quad \text{for every}\;\;m\geq m_0.\]
    By Thm. \ref{ns}, one has $\mathfrak{S}_{E^f_m}'\subseteq N\cap \Gamma_m'\subseteq \mathfrak{S}_{E^f_m}$ for $m\geq m_0$.
    For the inclusion $i \colon E_m\subset E_{m+1}$ with $i(1+\mathbb{K})\not \subset 1+\mathbb{K}$,
    one has $i(\mathfrak{S}_{E^f_m}')\not\subset \mathfrak{S}_{E^f_{m+1}}$ which leads to a contradiction
    \[N\cap\Gamma_{m+1}'\not \subset\mathfrak{S}_{E^f_{m+1}}.\]
    Thus,  a normal subgroup of $\Gamma_\infty'$ must be either $\{e\}$ or $\Gamma_\infty'$.
    \end{proof}

As in Sec. \ref{cte},
we identify the Fock  space \[\mathcal{F}(\mathbb{C}^n):=\mathbb{C}\delta_{v_0}\oplus \bigoplus_{k=1}^\infty (\mathbb{C}^n)^{\otimes k}\]
with the Hilbert space $\ell^2(E^f_n)$ by identifying $e_{\mu_1}\otimes\cdots\otimes e_{\mu_k}\in (\mathbb{C}^n)^{\otimes k}$ with $\delta_{\mu}\in \ell^2(E^f_n)$ ($\mu=\mu_1\cdots\mu_k\in E^f_n$).
The Fock representation $\Gamma_n\subset U(\mathcal{E}_n)\subset \mathbb{B}(\ell^2(E^f_n))$ remembers the action of $\Gamma_n\curvearrowright E^f_n\subset \partial E_n$ by the equation
\[g(\delta_\mu)=\delta_{g(\mu)},\quad \mu\in E^f_n,\;\; g\in \Gamma_n.\]
%Then, $\mathfrak{S}(E^f_n)\subset  (1_{\mathcal{E}_n}+\mathbb{K}(l^2(E^f_n)))$ holds.
The quotient map $\pi \colon \mathcal{E}_n\ni T_i\mapsto S_i\in \mathcal{O}_n=\mathcal{E}_n/\mathbb{K}(\ell^2(E^f_n))$ induces a surjective group homomorphism $\pi \colon \Gamma_n\to V_n$.
\begin{lem} \label{quotientGamma}
    $\operatorname{Ker}(\pi \colon \Gamma_n\to V_n)=\pi^{-1}(1_{\mathcal{O}_n})\cap \Gamma_n=(1_{\mathcal{E}_n}+\mathbb{K})\cap \Gamma_n=\mathfrak{S}_{E^f_n}.$
\end{lem}
\begin{proof}
    Every element in $\Gamma_n$ acts on $E^f_n=\{\delta_\mu\}_{\mu\in E^f_n}$ by a finite or infinite permutation.
If $g\in (1+\mathbb{K})\cap \Gamma_n$ acts on $E^f_n$ as an infinite permutation,
there are $\{\mu_k\}_{k=1}^\infty\subset E^f_n$ with $g(\mu_k)\not=\mu_k$ (i.e., $||(1-g)(\delta_{\mu_k})||_2=\sqrt{2}$).
Since $WOT-\lim_{k\to\infty}\delta_{\mu_k}=0$, the compact operator $(1-g)$ must satisfy $\lim_{k\to\infty}||(1-g)(\delta_{\mu_k})||_2=0$.
This is a contradiction, and one has $(1+\mathbb{K})\cap \Gamma_n\subset\mathfrak{S}_{E^f_n}$.
    
    %Thus, the inclusion $\Gamma_n\cap (1+\mathbb{K})\subset \mathfrak{S}(E^f_n)$ is obvious.
    Fix $h\in \mathfrak{S}_{E^f_n}$.
    Then, there is $N\in\mathbb{N}$ with \[h\in\mathfrak{S}_{\{v_0\}\cup\bigcup_{k=1}^N\{1, \cdots, n\}^k}=\mathfrak{S}_{\{\nu\in E^f_n\; \mid \; |\nu|\leq N\}}\subset \mathfrak{S}_{E^f_n}.\]
    Now one has
    \[h=(\sum_{|\mu|=N+1}T_\mu T_\mu^*)+(\sum_{|\nu|\leq N}T_{h(\nu)}e_nT_{\nu}^*)\in\Gamma_n\cap (1+\mathbb{K}).\]
\end{proof}
\begin{prop}\label{ex}
    The exact sequence
    \[1\to\mathfrak{S}_{E^f_n}\xrightarrow{i} \Gamma_n\xrightarrow{\pi} V_n\to 1\]
    induces the exact sequence
    \[\mathfrak{S}_{E^f_n}^{ab}\xrightarrow{i^{ab}} \Gamma_n^{ab}\xrightarrow{\pi^{ab}} V_n^{ab}\to 1.\]
\end{prop}
\begin{proof}
Fix $[g]^{ab}\in \operatorname{Ker}\pi^{ab}$ with a lift $g\in\Gamma_n$.
Since $\pi(g)\in V_n'$ and the surjectivity of $\pi$,
there are $x\in \Gamma_n'$ and $h\in\mathfrak{S}_{E^f_n}$ satisfying $g=hx$.
So we have $[g]^{ab}=i^{ab}([h]^{ab})$, which implies $\operatorname{Im} i^{ab}=\operatorname{Ker}\pi^{ab}$.
\end{proof}
\begin{lem}\label{cd}
    The following diagram commutes
    \[
    \xymatrix{
H_0(R_{E^f_n})\otimes\mathbb{Z}/2\mathbb{Z}\ar[r]^{1-n}\ar[d]^{\cong}&H_0(\mathcal{G}_{E_n})\otimes\mathbb{Z}/2\mathbb{Z}\ar[r]^-{\cong}&H_0(R_{\mathbb{N}}\times\mathcal{G}_{E_n})\otimes\mathbb{Z}/2\mathbb{Z}\ar[d]^{\zeta^s}\\
\mathfrak{S}_{E^f_n}^{ab}\ar[r]^{i^{ab}}&\Gamma_n^{ab}\ar[r]&[[R_\mathbb{N}\times \mathcal{G}_{E_n}]]^{ab},
    }\]
    where the right vertical map is the isomorphism in Thm. \ref{AH}, 3., and the left vertical map is the following isomorphism:
    \[H_0(R_{E^f_n})\otimes\mathbb{Z}/2\mathbb{Z}\ni (1_{\{v_0\}}+\operatorname{Im}\partial_1)\otimes\bar{1}\mapsto \left[(\sum_{i=1}^nT_{1i}T_{1i}^*)+(\sum_{i=2}^nT_iT_i^*)+e_nT_1^*+T_1e_n\right]^{ab}=[(1, v_0)]^{ab}\in \mathfrak{S}_{E^f_n}^{ab}.\]
\end{lem}
\begin{proof}
    Let 
    \[f_1 \colon H_0(R_{E^f_n})\otimes\mathbb{Z}/2\mathbb{Z}\xrightarrow{1-n}H_0(\mathcal{G}_{E_n})\otimes\mathbb{Z}/2\mathbb{Z}\xrightarrow{\cong} H_0(R_\mathbb{N}\times \mathcal{G}_{E_n})\xrightarrow{\zeta^s}[[R_\mathbb{N}\times\mathcal{G}_{E_n}]]^{ab},\]
    \[f_2 \colon H_0(R_{E^f_n})\otimes \mathbb{Z}/2\mathbb{Z}\xrightarrow{\cong}\mathfrak{S}_{E^f_n}^{ab}\xrightarrow{i^{ab}}\Gamma_n^{ab}\to [[R_\mathbb{N}\times\mathcal{G}_{E_n}]]^{ab}.\]
    By Example \ref{1-n}, the map $(1-n)$ sends $1_{\{v_0\}}+\operatorname{Im}\partial_1\in H_0(R_{E^f_n})$ to $1_{\{v_0\}}+\operatorname{Im}\partial_1\in H_0(\mathcal{G}_{E_n})$.
    Thus, one has 
    \begin{align*}
    &f_1(1_{\{v_0\}})\\
    =&\left[((1,2)\times \{v_0\})\sqcup ((2, 1)\times \{v_0\})\sqcup ((1,1)\times (\sqcup_{i=1}^nZ(i)))\sqcup ((2,2)\times (\sqcup_{i=1}^nZ(i)))\sqcup \bigsqcup_{k=3}^\infty((k, k)\times \partial E_n)\right]^{ab},\\
    &f_2(1_{\{v_0\}})\\
    =&\left[((1, 1)\times (Z(\{1\}, 1, 0, \{v_0\})\sqcup Z(\{v_0\}, 0, 1, \{1\})\sqcup (\sqcup_{i=1}^nZ(1i))\sqcup (\sqcup_{k=2}^nZ(k))))\sqcup\bigsqcup_{l=2}^\infty((l, l)\times \partial E_n)\right]^{ab}.\end{align*}
The following computation shows $f_1(1_{\{v_0\}})=f_2(1_{\{v_0\}})$
    \begin{align*}&\left(
\begin{array}{cccc}
     T_1(1-e_n)T_1^*+\sum_{k=2}T_kT_k^*+T_1e_n&e_n  \\
     e_nT_1^*&1-e_n 
\end{array}\right)\left(\begin{array}{ccc}
     \sum_{i=1}^nT_iT_i^*&e_n  \\
     e_n&\sum_{i=1}^nT_iT_i^* 
\end{array}\right)\\
&\quad\quad\quad\quad\quad\quad\quad\quad\quad\quad\quad\quad\quad\times\left(\begin{array}{cc}
     T_1(1-e_n)T_1^*+\sum_{k=2}T_kT_k^*+e_n T_1^*&T_1e_n  \\
     e_n&1-e_n 
\end{array}\right)\\
=&\left(\begin{array}{cc}
     T_1e_n+e_nT_1^*+\sum_{i=1}^nT_{1i}T_{1i}^*+\sum_{k=2}^nT_kT_k^*&0  \\
     0& 1 
\end{array}\right)\in U(\mathbb{M}_2(\mathcal{E}_n)).
    \end{align*}
\end{proof}

\begin{prop}\label{abg}
\begin{enumerate}
    \item For $2 \leq n<\infty$, the abelianization of $\Gamma_n$ is $\mathbb{Z}/2\mathbb{Z}$ and is computed by
    \[i^{ab} \colon \mathfrak{S}_{E^f_{2n}}^{ab}\cong \Gamma_{2n}^{ab}, \quad \pi^{ab} \colon \Gamma_{2n+1}^{ab}\cong V_{2n+1}^{ab}.\]
    \item There is an injective group homomorphism $\alpha \colon \Gamma_n\to V_{2n+1}$ and the abelianization is given by \[\Gamma_n\xrightarrow{\alpha}V_{2n+1}\to V_{2n+1}^{ab}.\]
    \item The natural map $\Gamma_n^{ab}\to\Gamma_{n+1}^{ab}$ is isomorphism for every $n\in \mathbb{N}$.
    \end{enumerate}
\end{prop}

\begin{proof}%[{Proof of Prop. \ref{abg}}]
    First,
    we show the statement 1.
    By Thm. \ref{AH} and Lem. \ref{cd},
    $\mathfrak{S}_{E^f_{2n}}^{ab}\xrightarrow{i^{ab}}\Gamma_{2n}^{ab}$ is injective,
    and $i^{ab} \colon \mathfrak{S}_{E^f_{2n}}\to \Gamma_{2n}^{ab}$ is bijective by Lem. \ref{ex}.

    By Lemma \ref{cd},
    the composition $\mathfrak{S}_{E^f_{2n+1}}\xrightarrow{i^{ab}}\Gamma_{2n+1}^{ab}\to [[R_\mathbb{N}\times \mathcal{G}_{E_n}]]^{ab}$ is zero. %and Lemma \ref{sai} shows $i^{ab}$ is a zero map.
    %We show that the map $\Gamma_{2n+1}^{ab}\to [[R(\mathbb{N})\times \mathcal{G}_{E_n}]]^{ab}$ is injective.
    %Fix $[U]\in\Gamma_{2n+1}^{ab}=[[\mathcal{G}_{E_{2n+1}}]]^{ab}$ which is zero in $[[R(\mathbb{N})\times\mathcal{G}_{E_n}]]^{ab}$.
    We show that 
    \[[U]^{ab}:=[(\sum_{i=1}^nT_{1i}T_{1i}^*)+(\sum_{k=2}^nT_iT_i^*)+e_nT_1^*+T_1e_n]^{ab}=i^{ab}([(1, v_0)]^{ab})=0\in \Gamma_{2n+1}^{ab}.\]
    There exist $N\in\mathbb{N}$ and $x_i, y_i\in [[R_{\{1, \cdots, N\}}\times\mathcal{G}_{E_n}]]$ satisfying
    \[\tilde{U}:=((1, 1)\times U)\sqcup \bigsqcup_{i=2}^N((i, i)\times \partial E_n)=\Pi_i[x_i, y_i]\in [[R_{\{1, \cdots, N\}}\times\mathcal{G}_{E_n}]]'.\]
    We take disjoint cylinder sets $\bigsqcup_{i=1}^NZ(\mu_i)\subset \partial E_n$.
    Consider a bisection 
    \[W:=\bigsqcup_{i=1}^N ((1, i)\times Z(Z(\mu_i), |\mu_i|, 0, \partial E_n)) \subset R_{\{1,\cdots, N\}}\times \mathcal{G}_{E_n}.\]
    %$W\subset R(\{1,\cdots, N\})\times \mathcal{G}_{E_n}$
    Then, we have
    \begin{enumerate}
        \item $s(W):=\bigsqcup_{i=1}^N\{i\}\times \partial{E_n}$, $r(W):=\{1\}\times (\bigsqcup_{i=1}^N Z(\mu_i))$,
        \item $W^{-1}W=\bigsqcup_{i=1}^N((i, i)\times\partial E_n)$ is the unit of  $[[R_{\{1,\cdots, N\}}\times\mathcal{G}_{E_n}]]$,
        \item The permutation $(\mu_11, \mu_1)\in\mathfrak{S}_{E^f_{2n+1}}$ appears in
        \[W\tilde{U}W^{-1}\sqcup Z((\bigsqcup_{i=1}^NZ(\mu_i))^c, 0, 0, (\bigsqcup_{i=1}^NZ(\mu_i))^c )=i((\mu_1 1, \mu_1))\in \Gamma_{2n+1}.\]
    \end{enumerate}
    The direct computation yields
    \begin{align*}
        &i((\mu_11, \mu_1))\\=&W\tilde{U}W^{-1}\sqcup Z((\bigsqcup_{i=1}^NZ(\mu_i))^c, 0, 0, (\bigsqcup_{i=1}^NZ(\mu_i))^c )\\
        =&W\Pi_i[x_i, y_i]W^{-1}\sqcup Z((\bigsqcup_{i=1}^NZ(\mu_i))^c, 0, 0, (\bigsqcup_{i=1}^NZ(\mu_i))^c )\\
=&\Pi_i[Wx_iW^{-1}, Wy_iW^{-1}]\sqcup Z((\bigsqcup_{i=1}^NZ(\mu_i))^c, 0, 0, (\bigsqcup_{i=1}^NZ(\mu_i))^c )\\
=&\Pi_i[(Wx_iW^{-1}\sqcup Z((\bigsqcup_{i=1}^NZ(\mu_i))^c, 0, 0, (\bigsqcup_{i=1}^NZ(\mu_i))^c )), (Wy_iW^{-1}\sqcup Z((\bigsqcup_{i=1}^NZ(\mu_i))^c, 0, 0, (\bigsqcup_{i=1}^NZ(\mu_i))^c ))]\\
\in &\Gamma_{2n+1}'.
    \end{align*}
    Thus, we have $i^{ab}([(1, v_0)]^{ab})=i^{ab}([(\mu_11, \mu_1)]^{ab})=[i((\mu_11, \mu_1))]^{ab}=0$ and $i^{ab}$ is zero.
    So, Lemma \ref{ex} shows $\pi^{ab} \colon \Gamma_{2n+1}^{ab}\cong V_{2n+1}^{ab}$.

    Next, we show  statement 2.
By the universality of $\mathcal{E}_n$,
there is a unital $*$-homomorphism 
\[\alpha \colon \mathcal{E}_n\ni T_i\mapsto S_i\in \mathcal{O}_{2n+1}\]
which is injective because $\alpha (e_n)=\sum_{i=n+1}^{2n+1}S_iS_i^*\not=0$.
Thus, we get an embedding $\alpha \colon \Gamma_n\to V_{2n+1}$ sending $T_1T_2^*+T_2T_1^*+\sum^n_{i=3}T_iT_i^*+e_n$ to the element $g_0=S_1S_2^*+S_2S_1^*+\sum_{i=3}^{2n+1}S_iS_i^*$.
Thus, Theorem \ref{AH} shows the map $\Gamma_n\xrightarrow{\alpha} V_{2n+1}\to V_{2n+1}^{ab}=\mathbb{Z}/2\mathbb{Z}$ is surjective.

 Next, we show statement 3.
Since $\Gamma_\infty=\bigcup_{n=2}^\infty\Gamma_n$,
it is enough to show the natural map $\Gamma_n^{ab}\to \Gamma_{2n+1}^{ab}$ is isomorphism for every $n\in \mathbb{N}$.
By the above argument on $\alpha$ and $\pi^{ab} \colon \Gamma_{2n+1}^{ab}\cong V_{2n+1}^{ab}$,
the commutative diagram
\[\xymatrix{
\Gamma_n^{ab}\ar[r]\ar[dr]^{\alpha}&\Gamma_{2n+1}^{ab}\ar[d]^{\pi^{ab}}\\
&V_{2n+1}^{ab}
}\]
proves the statement. 

%Finally, we show statement 4.   
\end{proof}
\begin{lem}\label{io}
If $g\in\Gamma_n$ commutes with every $\mathfrak{S}_{E^f_n}$,
then $g=e$.
\end{lem}
\begin{proof}
Fix a representation
\[g=\sum_{i=1}^MT_{\mu_i}T_{\nu_i}^*+\sum_{k=1}^NT_{v_k}e_n T_{w_k}^*.\]
Assume $\mu_1\not=\nu_1$.
If $v_k=w_k$ for $k=1, \cdots, N$,
one has  $h_k :=(w_k, \nu_1)\in\mathfrak{S}_{E^f_n}$ and
\[gh_k \colon w_k\mapsto \mu_1,\quad h_kg \colon w_k\mapsto \nu_1\not=\mu_1.\]
If $v_1\not=w_1$,
one has
\[gh_1 \colon w_1\mapsto \mu_1,\quad h_1g \colon w_1\mapsto v_1\not=\mu_1\; (\text{see Lem.} \;\ref{tabl}).\]
In both cases, $g$ does not satisfy the assumption, and one has $\mu_i=\nu_i$ for $i=1,\cdots, M$.

If $\mu_i=\nu_i$ for $i=1,\cdots, M$,
one has $g\in\operatorname{Ker}\pi=\mathfrak{S}_{E^f_n}$ and the assumption implies $g=e$.
%$gh$ sends $w_k$ to $\mu_1$ and $hg$ sends $w_k$  
\end{proof}
\begin{proof}[{Proof of Thm. \ref{ns}}]
    First, we consider the normal subgroup $N$ of $\Gamma_{2n}$.
    By the simplicity of $V_{2n}$,
    either $\pi(N)=V_{2n}$ or $N\subset \mathfrak{S}_{E^f_{2n}}$ holds.
    In the latter case, $N$ must be $\mathfrak{S}_{E^f_{2n}}$ or $\mathfrak{S}_{E^f_{2n}}'$.
    For $N$ with $\pi(N)=V_{2n}$,
    one has $N\cdot \mathfrak{S}_{E^f_{2n}}=\Gamma_{2n}$.
    
    If $N\cap \mathfrak{S}_{E^f_{2n}}=\{e\}$,
$\Gamma_{2n}\cong \mathfrak{S}_{E^f_{2n}}\times N$ holds and this is a contradiction by Lem. \ref{io}.

If $N\cap\mathfrak{S}_{E^f_{2n}}=\mathfrak{S}_{E^f_{2n}}$,
one has $N=N\cdot \mathfrak{S}_{E^f_{2n}}=\Gamma_{2n}$.

If $N\cap \mathfrak{S}_{E^f_{2n}}=\mathfrak{S}_{E^f_{2n}}'$,
the surjective map 
\[\mathbb{Z}/2\mathbb{Z}=\mathfrak{S}_{E^f_{2n}}/\mathfrak{S}_{E^f_{2n}}'\to \Gamma_{2n}/N=(\mathfrak{S}_{E^f_{2n}}\cdot N)/N\]
is injective, and one has $N=\Gamma_{2n}'$ by Prop. \ref{abg}.
So, the normal subgroup $N\subset \Gamma_{2n}$ is either $\Gamma_{2n}', \mathfrak{S}_{E^f_{2n}}, \mathfrak{S}_{E^f_{2n}}'$.

Next, we consider the normal subgroup $N$ in $\Gamma_{2n+1}$.
By Prop. \ref{abg},
one has $\mathfrak{S}_{E^f_{2n+1}}\subset \Gamma_{2n+1}'$.
The normal subgroup of $V_{2n+1}$ is either $\{e\}, V_{2n+1}', V_{2n+1}$,
and if $\pi(N)=\{e\}$, then $N=\mathfrak{S}_{E^f_{2n+1}}$ or $\mathfrak{S}_{E^f_{2n+1}}'$.

Consider the case $\pi(N)=V_{2n+1}$ (i.e., $N\cdot\mathfrak{S}_{E^f_{2n+1}}=\Gamma_{2n+1}$).

If $N\cap \mathfrak{S}_{E^f_{2n+1}}=\{e\}$,
Lem. \ref{io} and $\Gamma_{2n+1}\cong N\times\mathfrak{S}_{E^f_{2n+1}}$ give a contradiction.

If $N\cap \mathfrak{S}_{E^f_{2n+1}}=\mathfrak{S}_{E^f_{2n+1}}'$,
the surjection
\[\mathfrak{S}_{E^f_{2n+1}}/\mathfrak{S}'_{E^f_{2n+1}}\to (N\cdot \mathfrak{S}_{E^f_{2n+1}})/N=\Gamma_{2n+1}/N\]
is injective.
However, this implies $N=\Gamma_{2n+1}'\supset\mathfrak{S}_{E^f_n}$ by Prop. \ref{abg} and makes a contradiction.
Thus, one has $N\cap \mathfrak{S}_{E^f_{2n+1}}=\mathfrak{S}_{E^f_{2n+1}}$ and 
$N=N\cdot \mathfrak{S}_{E^f_{2n+1}}=\Gamma_{2n+1}$.

Consider the case $\pi(N)=V_{2n+1}'$ where one has
\[N\subset \Gamma_{2n+1}'\cdot\mathfrak{S}_{E^f_{2n+1}}\subset\Gamma_{2n+1}'\subset N\cdot\mathfrak{S}_{E^f_{2n+1}}, \quad N\cdot\mathfrak{S}_{E^f_{2n+1}}=\Gamma_{2n+1}'.\]

If $N\cap \mathfrak{S}_{E^f_{2n+1}}=\{e\}$,
Lem. \ref{io} and $\Gamma_{2n+1}'\cong N\times\mathfrak{S}_{E^f_{2n+1}}$ give a contradiction.

If $N\cap \mathfrak{S}_{E^f_{2n+1}}=\mathfrak{S}_{E^f_{2n+1}}$,
one has $N=\Gamma_{2n+1}'$.

If $N\cap\mathfrak{S}_{E^f_{2n+1}}=\mathfrak{S}_{E^f_{2n+1}}'$,
the surjection
\[ \langle [(1, v_0)]^{ab} \rangle=\mathbb{Z}/2\mathbb{Z}=\mathfrak{S}_{E^f_{2n+1}}/\mathfrak{S}_{E^f_{2n+1}}'\to N\cdot \mathfrak{S}_{E^f_{2n+1}}/N=\Gamma_{2n+1}'/N\]
is injective and this implies $[U]^{ab}=[(1, v_0)]^{ab}\not=e\in (\Gamma_{2n+1}')^{ab}$.
Recall the notation $U\in \Gamma_{2n+1}, \;\;\tilde{U}\in [[R_\mathbb{N}\times\mathcal{G}_{E_n}]]$ used in Prop. \ref{abg}
\[(1, v_0)=U:=(\sum_{i=1}^{2n+1}T_{1i}T_{1i}^*)+(\sum_{i=2}^{2n+1}T_iT_i^*)+e_{2n+1}T_1^*+T_1e_{2n+1}\in\mathfrak{S}_{E^f_{2n+1}}\subset\Gamma_{2n+1}'.\]
%(see proof of Prop. \ref{abg}).

By \cite[Cor. 6.10.]{XinLi},
    we have 
    \[\Gamma_{2n+1}'\ni U\mapsto [\tilde{U}]^{ab}=0\in ([[R_\mathbb{N}\times \mathcal{G}_{E_{2n+1}}]]')^{ab}=0.\]
    %(see the proof of Prop. \ref{abg} for the notation $\tilde{U}$).
    The same argument as in the proof of Prop. \ref{abg}, 1. shows $[U]^{ab}=e\in (\Gamma_{2n+1}')^{ab}$.
    This is a contradiction and $N\cap \mathfrak{S}_{E^f_{2n+1}}$ must be $\mathfrak{S}_{E^f_{2n+1}}$.

Thus, the normal subgroup $N$ is either $\Gamma_{2n+1}', \mathfrak{S}_{E^f_{2n+1}}, \mathfrak{S}_{E^f_{2n+1}}'$.

The same argument shows that the normal subgroup of $\Gamma_n'$ is either $\mathfrak{S}_{E^f_n}$, $\mathfrak{S}_{E^f_n}'$.
\end{proof}

%%%%%%%%%%%%%%%%%%%%%%%%%%%%%%%%%%%%%%%%%%%%%%%%%%%%%%%%%%%%%%%%%%%%%%%%%%%%%%%%%%%%%%%%%%%%%%%%%%%%%%%%%%%%%%%%%%%%%%%%%%%%%%%%%%%%%%%%%%%%%%%%%%%%%%%%%%%%%%%%%%%%%%%%%%%%%%%%%%%%%%%%%%%%%%%%%%%%%%%%%%%%%%%%%%%%%%%%%%%%%%%%%%%%%%%%%%%%%%%%%%%%%%%%%%%%%%%%%%%%%%%%%%%%%%%%%%%%%%%%%%%%%%%%%%%%%%%%%%%%%%%%%%%%%%%%%%%%%%%%%%%%%%%%%%%%

\subsection{$\gamma_{c^f_n}$-KMS states of $C(\partial O_n) \rtimes _r V_n$}
%Neshveyev, full not=reduced,
We write $m \colon C(\partial O_n)\ni f\mapsto \int_{\partial O_n} f(x) dm(x)\in \mathbb{C}$ where $m$ is the product measure $\bigotimes_{i=1}^\infty(\sum_{j=1}^n\frac{1}{n}\delta_j)$ (i.e., $m$ is the composition $C(\partial O_n)\subset M_{n^\infty}\xrightarrow{\text{trace}}\mathbb{C}$),
and let
\[E \colon C(\partial O_n)\rtimes_r V_n\ni f\lambda_g\mapsto \delta_{e, g} f\in C(\partial O_n)\] the canonical conditional expectation.

%Combining the argument in \cite{OP} and the unique trace property of C*-simple groups, we will show the following theorem.
\begin{thm}\label{M}
For $C(\partial O_n)\rtimes_r V_n$,
there is a $\gamma_{c^f_n}\text{-KMS}_\beta$-state if and only if $\beta={\operatorname{log} n}$,
and the KMS state is unique, which is given by
\[\psi \colon C(\partial O_n)\rtimes_r V_n\xrightarrow{E} C(\partial O_n)\xrightarrow{m}\mathbb{C}.\]
\end{thm}
%%%%%%%%%%%%%%%%%%%%%%%%%%%%%%%%%%%%%%%%%%%%%%%%%%%%%%%%%%%%%%%%%%%%%%%%%%%%%%%%%%%%%%%%%%%%%%%%%%%%%%%%%%%%%%%%%%%%%%%%%%%%%%%%
%Site A, short proof
\begin{comment}
\begin{lem}\label{stabilizer}
For $x=x_1x_2\cdots\in \partial O_n\backslash \{\alpha\nu\nu\nu\cdots\in\partial O_n\;|\; \alpha, \nu\in\{1, \cdots, n\}^*\}$ (i.e., $x$ is not eventually periodic),
the stabilizer subgroup $(V_n)_x:=\{g\in V_n\;|\; g(x)=x\}$ is the union of C*-simple groups :
\[(V_n)_x=\bigcup_{n=1}^\infty \operatorname{Rist}_{V_n}(Z^\infty(x_1x_2\cdots x_n)^c)\]
\end{lem}
\begin{proof}

\end{proof}
\end{comment}

Note that for each $x \in \partial O_n \setminus \{\alpha\nu\nu\nu\cdots\in\partial O_n\;|\; \alpha, \nu\in\{1, \cdots, n\}^*\}$, one has $\displaystyle (V_n)_x = \bigcup _{k=1} \operatorname{Fix}_{V_n}(Z^{\infty}(x_1 \cdots x_k))$ since $x$ is not eventually periodic.

\begin{prop}\label{uniquetr}
    For each $\mu \in \{ 1, 2, \ldots n \} ^*$, the action $\operatorname{Fix}_{V_n}(Z^{\infty}(\mu)) \curvearrowright \partial O_n \setminus Z^{\infty}(\mu)$ is a faithful boundary action.
    In particular, each $\operatorname{Fix}_{V_n}(Z^{\infty}(\mu))$ has the unique trace property.
\end{prop}

\begin{proof}
    There exists $g \in V_n$ such that $g(Z^{\infty}(\mu)) = Z^{\infty}(1) \sqcup \cdots \sqcup Z^{\infty}(n-1)$.
    Then, taking an adjoint by $g$, one can identify $\operatorname{Fix}_{V_n}(Z^{\infty}(\mu)) \curvearrowright \partial O_n \setminus Z^{\infty}(\mu)$ with $\operatorname{Fix}_{V_n}(Z^{\infty}(1) \sqcup \cdots \sqcup Z^{\infty}(n-1)) \curvearrowright Z^{\infty}(n)$, which is isomorphic to $V_n \curvearrowright \partial O_n$.
\end{proof}

\begin{cor}\label{ari}
    For each $x \in \partial O_n \setminus \{\alpha\nu\nu\nu\cdots\in\partial O_n\;|\; \alpha, \nu\in\{1, \cdots, n\}^*\}$, $(V_n)_x$ has the unique trace property.
\end{cor}

\begin{proof}
    Let $N$ be a normal amenable subgroup of $(V_n)_x$.
    Since $N \cap \operatorname{Fix}_{V_n}(Z^{\infty}(x_1 \cdots x_k))$ is a normal amenable subgroup of $\operatorname{Fix}_{V_n}(Z^{\infty}(x_1 \cdots x_k))$ and $\operatorname{Fix}_{V_n}(Z^{\infty}(x_1 \cdots x_k))$ has the unique trace property, $N \cap \operatorname{Fix}_{V_n}(Z^{\infty}(x_1 \cdots x_k)) = \{ e \}$.
    Thus $\displaystyle N = \bigcup_{k} N \cap \operatorname{Fix}_{V_n}(Z^{\infty}(x_1 \cdots x_k)) = \{ e \}$ and this implies that $(V_n)_x$ has the unique trace property.
\end{proof}

\begin{proof}[{Proof of Thm. \ref{M}}]
Let $\mu$ be a quasi-invariant measure of $\partial O_n\rtimes V_n$ with Radon-Nykodim cocycle $e^{-\beta c_n}$.
By the definition of topological full group $[[\mathcal{G}_n]]=V_n$,
every local homeomorphism of $\partial O_n$ given by a bisection of $\partial O_n\rtimes V_n$ is a local homeomorphism given by a bisection of $\mathcal{G}_n$.
Thus, \cite{OP, Nes} shows $\beta=\log n$ and $\mu=m$.
So \cite{Nes} shows that $\psi$ is a $c^f_n$-$\mathrm{KMS}_{\log n}$ state and $c^f_n$-KMS$_\beta$ state exists only for $\beta=\log n$.

We will show that every $c^f_n$-KMS$_{\log n}$ state is equal to $\psi$. 
For a $c^f_n$-KMS$_{\log n}$ state $\varphi : C(\partial O_n)\rtimes_r V_n\to \mathbb{C}$,
we write
\[\tilde{\varphi} : C(\partial O_n)\rtimes V_n\to C(\partial O_n)\rtimes_r V_n\xrightarrow{\varphi}\mathbb{C}.\]
By \cite{Nes}, there exists a measurable field $\{m, \{\tilde{\tau}_x\}_{x\in\partial O_n}\}$ such that
\[\tilde{\varphi}=\int_{\partial O_n}\tilde{\tau}_x dm(x),\]
where $\tilde{\tau}_x : C^*((\partial O_n\rtimes V_n)^x_x)\to \mathbb{C}$ is a tracial state.
Since $\tilde{\varphi}$ factor through the reduced groupoid C*-algebra $C(\partial O_n)\rtimes_r V_n$,
\cite[Prop. 2.10, 3.1]{CNes} show that $\tilde{\tau}_x$ is a pull back of a tracial state $\tau_x : C^*_r((V_n)_x)\to\mathbb{C}$ for $m$-a.e. $x\in\partial O_n$.
Since the set of  eventually periodic words $\{\alpha\nu\nu\nu\cdots\in \partial O_n\;|\; \alpha, \nu\in \{1, \cdots, n\}^*\}$ is countable $m$-null set,
we may assume that $\tilde{\tau}_x : C^*((V_n)_x)\to C^*_r((V_n)_x)\xrightarrow{\tau_x}\mathbb{C}$ holds for $m$-a.e. $x\in\partial O_n$ where $x$ is not eventually periodic. 
By Cor. \ref{ari},
each $\tau_x$ must be the canonical trace, and one has
\[\tau_x : C^*_r((V_n)_x)\subset C^*_r(V_n)\xrightarrow{\text{canonical trace}}\mathbb{C},\]
and this implies $\varphi=\psi$.

Since the boundary set (see \cite[p271]{LLNes}) of $\partial O_n$ of the cocycle $c_n$ (and $c^f_n$) is $\emptyset$,
\cite[Thm. 1.4.]{LLNes} shows that there are no $\gamma_{c^f_n}$-ground states.
\end{proof}

\begin{remark}
    In contrast to the case of $\mathcal{O}_n$ and the groupoid C*-algebras with trivial isotropy,
    the fixed point algebra $(C(\partial O_n) \rtimes _r V_n)^{\mathbb{T}}$ seems to be complicated, and the isotropy of the groupoid $\partial O_n\rtimes V_n$ is large (almost the same as $V_n$).
Thus, one can not simply apply the previous results on the KMS-states of groupoid C*-algebras and the key ingredients in the above theorem are the results \cite{LBMB} and \cite{BKKO} on the C*-simplicity and unique trace property.
\end{remark}

%%%%%%%%%%%%%%%%%%%%%%%%%%%%%%%%%%%%%%%%%%%%%%%%%%%%%%%%%%%%%%%%%%%%%%%%%%%%%%%%%%%%%%%%%%%%%%%%%%%%%%%%%%%%%%%%%%%%%%%%%%%%%%%%%%%%%%%%%%%%%%%%%%%%%%%%%%%%%%%%%%%%%%%%%%%%%%%%%%%%%%%%%%%%%%%%%%%%%%%%%%%%%%%%%%%%%%%%%%%%%%%%%%%%%%%%%%%%%%%%%%%%%%%%%%%%%%%%%%%%%%%%%%%%%%%%%%%%%%%%%%%%%%%%%%%%%%%%%%%%%%%%%%%%%%%%%%%%%%%%%%%%%%%%%%%%

\subsection{$\gamma_{d^f_\infty}$-Ground states of $C(\partial E_\infty)\rtimes_r \Gamma_\infty$}
\begin{thm}\label{MM}
    A state $\psi \colon C(\partial E_\infty)\rtimes_r\Gamma_\infty\to\mathbb{C}$ is a ground state for the $\mathbb{R}$-action $\gamma_{d^f_\infty}$ if and only if $\psi$ is given by
    \[\psi \colon C(\partial E_\infty)\rtimes_r\Gamma_\infty\xrightarrow{E}C(\partial E_\infty)\rtimes_r (\Gamma_\infty)_{v_0}\xrightarrow{ev_{v_0}}C^*_r((\Gamma_\infty)_{v_0})\xrightarrow{\varphi}\mathbb{C},\]
    \[\psi (f\lambda_g)=f(v_0)\varphi (1_{(\Gamma_\infty)_{v_0}}(g)\lambda_g),\]
    where $(\Gamma_\infty)_{v_0}:=\{g\in\Gamma_\infty\;|\; g(v_0)=v_0\}$ is the stabilizer of $v_0$, $E(f\lambda_g)=1_{(\Gamma_\infty)_{v_0}}(g)f\lambda_g$ is the conditional expectation, and $\varphi$ is a state of $C^*_r((\Gamma_\infty)_{v_0})$.
\end{thm}
{\color{red} }
\begin{proof}
    First, we show that the state $\psi$ in the theorem is a ground state.
    Since $C(\partial E_\infty)=\overline{\operatorname{span}}\{1, T_\mu T_\mu^*\;|\;\mu\in E^f_\infty\}$ and $\Gamma_\infty=\bigcup_{n=2}^\infty\Gamma_n$,
    it is enough to show
    \[|\psi (b\gamma_{d^f_\infty}(z)(\sum_U a_U 1_U\lambda_{g_U}))|\leq ||b||||\sum_U a_U 1_U\lambda_{g_U}||\]
    for $b\in C(\partial E_\infty)\rtimes_r\Gamma_\infty$, $\operatorname{Im}(z)\geq 0,$ clopen sets $U\subset\partial E_\infty, a_U\in\mathbb{C}$ and $g_U \in \Gamma _n$.

    For the clopen set $U$ with $v_0\in g_U^{-1}(U)$,
    we may assume that $1_U\lambda_{g_U}=1_{Z(\mu_U)\backslash \cup_{i=1}^NZ(\mu_U\nu_i)}\lambda_{g_U}$ with 
    \[g_U=T_{\mu_U} (1-\sum_{i=1}^NT_{\nu_i}T_{\nu_i}^*)+\cdots\in\Gamma_\infty, \;\text{for}\; \mu_U\in E^f_\infty.\]
    Note that $\gamma_{d^f_\infty}(z)(1_U\lambda_{g_U})=e^{-|\mu_U|\operatorname{Im}(z)}e^{i|\mu_U|\operatorname{Re}(z)}1_U\lambda_{g_U}$ holds for $U$ with $v_0\in g_U^{-1}(U)$.

Since $\psi(f)=f(v_0)$, the subalgebra $C(\partial E_\infty)$ is in the multiplicative domain of $\psi$,
and one has
\begin{align*}
    \psi (b\gamma_{d^f_\infty}(z)(\sum_U a_U 1_U\lambda_{g_U}))=&\sum_{U,\; v_0\not\in g_U^{-1}(U)}a_U\psi (b\gamma_{d^f_\infty}(z)(1_U\lambda_{g_U}))\psi(1_{g_U^{-1}(U)})\\
    &+\sum_{U,\; v_0\in g_U^{-1}(U)}a_U\psi(b1_U\lambda_{g_U})e^{-|\mu_U|\operatorname{Im}(z)}e^{i|\mu_U|\operatorname{Re}(z)}\\
    =&\sum_{U,\; v_0\in g_U^{-1}(U)}a_U\psi(b1_U\lambda_{g_U})e^{-|\mu_U|\operatorname{Im}(z)}e^{i|\mu_U|\operatorname{Re}(z)}.
\end{align*}
    Thus, the function $\psi (b\gamma_{d^f_\infty}(z)(\sum_U a_U 1_U\lambda_{g_U}))$ is bounded on $\{z\in\mathbb{C}\;|\; \operatorname{Im}(z)\geq 0\}$ and the Phragmen–Lindel\"{o}f theorem shows
    \begin{align*}
        |\psi (b\gamma_{d^f_\infty}(z)(\sum_U a_U 1_U\lambda_{g_U}))|&\leq \operatorname{sup}_{t\in\mathbb{R}}|\psi (b\gamma_{d^f_\infty}(t)(\sum_U a_U 1_U\lambda_{g_U})|\\
        &\leq \operatorname{sup}_{t\in\mathbb{R}} ||b||||\gamma_{d^f_\infty}(t)(\sum_U a_U 1_U\lambda_{g_U})||\\
        &=||b||||\sum_U a_U 1_U\lambda_{g_U}||.
    \end{align*}

    Second, we show that $\psi=\psi | _{\mathrm{C}^*_r((\Gamma _{\infty})_{v_0})}\circ ev_{v_0}\circ E$ holds for an arbitrary ground state $\psi$.

%%%%%%%%%%%%%%%%%%%%%%%%%%%%%%%%%%%%%%%%%%%%%%%%%%%%%%%%%%%%%%%%%%%%%%%%%%%%%%%%%%%%%%%%%%%%%%%%%%%%%%%%%%%%%%%%%%%%%%%%%%%%%%%%%%%%%%

For the groupoid $\partial E_\infty\rtimes \Gamma_\infty$ and the cocycle $d^f_\infty$,
the boundary set of the cocycle $d^f_\infty$ (see \cite{LLNes}) is the singleton $\{v_0\}$ and the boundary groupoid of $d^f_\infty$ is the group $(\Gamma_\infty)_{v_0}$.
Thus, \cite[Thm. 1.4]{LLNes} implies that the pull-back \[\tilde{\psi} : C(\partial E_\infty)\rtimes \Gamma_\infty\to C(\partial E_\infty)\rtimes_r\Gamma_\infty\xrightarrow{\psi}\mathbb{C}\]
satisfies $\psi(f\lambda_g)=\tilde{\psi}(f\lambda_g)=f(v_0)\tilde{\psi} |_{C^*((\Gamma_\infty)_{v_0})}(E(\lambda_g))=\psi |_{C^*_r((\Gamma_\infty)_{v_0})}(ev_0(E(f\lambda_g)))$ for $f\in C(\partial E_\infty)$ and $g\in\Gamma_\infty$.

\end{proof}
\begin{remark}
    For $\beta<\infty$,
    the KMS condition implies $\psi(1_{Z(i)})=0$ and $e^{-\beta}=e^{-\beta}\psi(1-1_{Z(i)})=\psi(1_{Z(i)}-1_{Z(ii)})=0$.
    Thus, there are no $\gamma_{d_\infty^f}$-KMS states for $\beta<\infty$.
\end{remark}

%%%%%%%%%%%%%%%%%%%%%%%%%%%%%%%%%%%%%%%%%%%%%%%%%%%%%%%%%%%%%%%%%%%%%%%%%%%%%%%%%%%%%%%%%%%%%%%%%%%%%%%%%%%%%%%%%%%%%%%%%%%%%%%%%%%%%%%%%%%%%%%%%%%%%%%%%%%%%%%%%%%%%%%%%%%%%%%%%%%%%%%%%%%%%%%%%%%%%%%%%%%%%%%%%%%%%%%%%%%%%%%%%%%%%%%%%%%%%%%%%%%%%%%%%%%%%%%%%%%%%%%%%%%%%%%%%%%%%%%%%%%%%%%%%%%%%%%%%%%%%%%%%%%%%%%%%%%%%%%%%%%%%%%%%%%%%%%

\subsection{$\gamma_{d^f_n}$-KMS states of  $C(\partial E_n)\rtimes_r\Gamma_n$}
In this section,
we characterize the $\gamma_{d^f_n}$-KMS-states on $C(\partial E_n) \rtimes _r \Gamma _n$.

%As a conclusion these $\gamma_{d^f_n}-KMS_\beta-$states ($\infty> \beta\geq \operatorname{log} n$) are in one to one correspondence to the  traces of $C^*_r(\mathfrak{S}_{E^f_n})$ (see Thm. \ref{yes}).

In the case of $\beta=\infty$,
the same argument as in the proof of Thm. \ref{MM} shows the following.
\begin{thm}\label{MMM}
    A state $\psi \colon C(\partial E_n)\rtimes_r\Gamma_n\to\mathbb{C}$ is a $\gamma_{d^f_n}$-ground state if and only if $\psi$ is given by
    \[\psi \colon C(\partial E_n)\rtimes_r\Gamma_n\xrightarrow{E}C(\partial E_n)\rtimes_r(\Gamma_n)_{v_0}\xrightarrow{ev_{v_0}}C^*_r((\Gamma_n)_{v_0})\xrightarrow{\varphi}\mathbb{C},\]
    \[\psi (f\lambda_g)=f(v_0)\varphi(1_{(\Gamma_n)_{v_0}}(g)\lambda_g), \quad f\in C(\partial E_n),\quad g\in\Gamma_n,\]
    where $(\Gamma_n)_{v_0}$ is the stabilizer of $v_0$, $E(f\lambda_g):=1_{(\Gamma_n)_{v_0}}(g) f\lambda_g$ is a conditional expectation and $\varphi$ is a state of $C^*_r((\Gamma_n)_{v_0})$.
\end{thm}

For $\beta < \infty$, we obtain the following.
\begin{thm}\label{kms}
For $\beta<\infty$,
    there exist $\gamma_{d^f_n}\text{-KMS}_\beta$-states of $C(\partial E_n)\rtimes_r\Gamma_n$ if and only if $\beta\geq \operatorname{log} n$.
    \begin{enumerate}
        \item For $\beta >\operatorname{log} n$,
        the KMS$_\beta$-state is given by
        \[\psi (f\lambda_g) :=\sum_{\mu\in E^f_n} (1-ne^{-\beta})e^{-\beta|\mu|}f(\mu) 1_{(\Gamma_n)_\mu}(g) \tau (\lambda_{g_\mu^{-1}gg_\mu}), \quad f\in C(\partial E_n),\;\; g\in\Gamma_n,\]
        where $1_{(\Gamma_n)_\mu}$ is the characteristic function of the stabilizer subgroup $(\Gamma_n)_\mu=\{g\in \Gamma_n\;|\; g(\mu)=\mu\}$, $g_\mu$ is an element satisfying $g_\mu(v_0)=\mu$, and $\tau \colon C^*_r((\Gamma_n)_{v_0})\to \mathbb{C}$ is a trace.
        
        The above presentation does not depend on the choice of $g_\mu$, and there is a one-to-one correspondence between $\gamma_{d^f_n}\text{-KMS}_{\beta}$-states on $C(\partial E_n)\rtimes_r \Gamma_n$ and tracial states of $C^*_r((\Gamma_n)_{v_0})$.
        {\color{red}}
        \item For $\beta=\operatorname{log} n$,
        the KMS state is given by
        \[\psi \colon C(\partial E_n)\rtimes_r\Gamma_n\xrightarrow{E}C(\partial E_n)\rtimes_r\mathfrak{S}_{E^f_n}\xrightarrow{Ev_{E^\infty_n}} C(E^\infty_n)\rtimes_r\mathfrak{S}_{E^f_n}=C(E^\infty_n)\otimes C^*_r(\mathfrak{S}_{E^f_n})\xrightarrow{m\otimes\tau}\mathbb{C},\]
        \[\psi (1_{Z^\infty(\mu)\lambda_g})=m(Z^\infty(\mu))\tau(1_{\mathfrak{S}_{E^f_n}}(g)\lambda_g),\]
        where we write \[Ev_{E^\infty_n} \colon C(\partial E_n)\rtimes_r\mathfrak{S}_{E^f_n}\ni 1_{Z(\mu)}\lambda_g\mapsto 1_{Z^\infty(\mu)}\lambda_g\in C(E^\infty_n)\rtimes_r\mathfrak{S}_{E^f_n},\]\[E(f \lambda_g)=1_{\mathfrak{S}_{E^f_n}}(g)f\lambda_g,\]
        and $\tau$ is a tracial state of $C^*_r(\mathfrak{S}_{E^f_n})$. 
        {\color{red}}
        %satisfying the following condition $(\flat)$:
        %\[ (\flat) \quad\quad\quad \tau ({_\mu g_\nu}^{-1}\cdot 1_{Z^\infty(\mu)}\otimes \lambda_{{_\mu g_\nu}^{-1}h{_\mu g_\nu}})=n^{|\mu|-|\nu|}\tau (1_{Z^\infty(\mu)}\otimes \lambda_h), \quad\quad\quad\quad\quad\quad\quad\quad\quad\quad\quad\]\[\text{for any}\;\;\mu\in E^f_n\backslash \{v_0\},\, \; h\in\mathfrak{S}_{E^f_n},\]\[ \text{and the element of the form}\;\; _\mu g_\nu=T_\mu T_\nu^*+\cdots \in \Gamma_n.\]
            \end{enumerate}
\end{thm}

For $\beta=\operatorname{log} n$,
we need the following lemmas.
\begin{lem}\label{umu}
Let $\psi \colon C(\partial E_n)\rtimes_r\Gamma_n\to\mathbb{C}$ be the $\gamma_{d^f_n}\text{-KMS}_{\operatorname{log}n}$-state of $C(\partial E_n)\rtimes_r\Gamma_n$.
    For $\mu\in E^f_n\backslash\{v_0\}$ and the subgroup $\pi^{-1}(\operatorname{Rist}_{V_n}(Z^\infty(\mu)^c))\subset \Gamma_n$,
    the state
    \[C^*_r(\pi^{-1}(\operatorname{Rist}_{V_n}(Z^\infty(\mu)^c)))\ni\lambda_g\mapsto \frac{\psi(1_{Z(\mu)} \lambda_g)}{\psi (1_{Z(\mu)})}=n^{|\mu|}\psi (1_{Z(\mu)}\lambda_g)\in\mathbb{C}\]
    is tracial.
    In particular,
    we have $\psi (1_{Z(\mu)}\lambda_g)=1_{\mathfrak{S}_{E^f_n}}(g)\psi (1_{Z(\mu)}\lambda_g)$.
\end{lem}
\begin{proof}
    By considering a pull-back and applying \cite[Thm. 1.3]{Nes},
    one has $\psi (1_{Z(\mu)})=e^{-\beta|\mu|}=n^{-|\mu|}$ 
    and $C_0(E^f_n)\subset \operatorname{Ker} \psi$ for $\beta=\operatorname{log} n$ because of the form of the quasi-invariant measure with the Radon--Nikodym derivative $n^-{d^f_n}$ (see also the proof of Thm. \ref{kms}).

    An arbitrary element  of $ \operatorname{Rist}_{V_n}(Z^\infty(\mu)^c)$ is given by
    %we may assume that %$\bigsqcup_{i\in I}Z^\infty(\mu_i)=Z^\infty(\mu)$ by taking a suitable subdivision $S_{\mu_i}S_{\nu_i}^*=\sum_{|\eta|=|\mu|}S_{\mu_i}S_\eta S_\eta^*S_{\nu_i}^*$.
    %Since \[\bar{g}=\sum_{i\in I}S_{\mu_i}S_{\nu_i}^*+\sum_{j\not\in I}S_{\mu_j}S_{\nu_j}^*\] must point-wisely  fix \[Z^\infty(\mu)=\bigsqcup_{i\in I}Z^\infty(\mu_i),\]
    %one has $\mu_i=\nu_i$ for $i\in I$ which implies
    \[S_\mu S_\mu^*+\sum_{j=2}^M S_{\mu_j} S_{\nu_j}^*\in V_n,\]
    %by Lem. \ref{ut},
    and there is a lift
    \[_\mu g_\mu:=T_\mu T_\mu^*+\sum_{j=2}^MT_{\mu_j}T_{\nu_j}^*+\sum_{k=1}^NT_{v_k}e_nT_{w_k}^*\in \Gamma_n.\]
    Thus, every element of $\pi^{-1}(\operatorname{Rist}_{V_n}(Z^\infty(\mu)^c))$ is represented by
    \[(T_\mu T_\mu^*+\cdots) h={_\mu g_\mu} h\in\Gamma_n,\quad h\in \mathfrak{S}_{E^f_n}.\]
For $h\in\mathfrak{S}_{E^f_n}$,
the set $\operatorname{supp}(h):=\{\mu\in E^f_n \;|\; h(\mu)\not=\mu\}$ is finite, and one has 
\begin{align*}
    \gamma_{d^f_n}(i\beta)(1_{Z(\mu)}\lambda_{{_\mu g_\mu} h})=&\gamma_{d^f_n}(i\beta)(\lambda_{_\mu g_\mu}1_{Z(\mu)}\lambda_h)\\
    =&\gamma_{d^f_n}(i\beta)(\lambda_{_\mu g_\mu}1_{Z(\mu)\cap \operatorname{supp}(h)^c}\lambda_h)+\gamma_{d^f_n}(i\beta)(\lambda_{_\mu g_\mu} 1_{\operatorname{supp(h)}}\lambda_h)\\
    \in &(\lambda_{_\mu g_\mu}1_{Z(\mu)\cap \operatorname{supp}(h)^c}\lambda_h) + C_0(E^f_n)\lambda_{{_\mu g_\mu}h}\\
    = &1_{Z(\mu)}\lambda_{{_\mu g_\mu} h}+\lambda_{{_\mu g_\mu}h}C_0(E^f_n).
\end{align*}
    
    Thus, the KMS condition and $C_0(E^f_n)\subset \operatorname{Ker}\psi$ yield
    \begin{align*}
        \psi(1_{Z(\mu)}\lambda_{{_\mu g^1_\mu}h_1}\lambda_{{{_\mu g^2_\mu}}h_2})=&\psi (\lambda_{{_\mu g^2_\mu}h_2}1_{Z(\mu)}\lambda_{{_\mu g^1_\mu}h_1})\\
        =&\psi (\lambda_{_\mu g^2_\mu}(1_{Z(\mu)}\lambda_{h_2}+C_0(E^f_n)\lambda_{h_2})\lambda_{{_\mu g^1_\mu}h_1})\\
        =&\psi (1_{Z(\mu)}\lambda_{{_\mu g^2_\mu}h_2}\lambda_{{_\mu g^1_\mu}h_1})
    \end{align*}
(i.e., $\lambda_g\mapsto \frac{\psi(1_{Z(\mu)} \lambda_g)}{\psi (1_{Z(\mu)})}$ is a trace).
    Since $\operatorname{Rist}_{V_n}(Z^\infty(\mu)^c)$ has the unique trace property by Lem. \ref{uniquetr},
    $\mathfrak{S}_{E^f_n}$ is a maximal normal amenable subgroup of $\pi^{-1}(\operatorname{Rist}_{V_n}(Z^\infty(\mu)^c))$, and \cite[Thm. 4.1.]{BKKO}, $\psi (1_{Z(\mu)}\lambda_g)=1_{\mathfrak{S}_{E^f_n}}(g)\psi (1_{Z(\mu)}\lambda_g)$.
\end{proof}
\begin{lem}\label{ummu}
Let $\psi$ be the $\gamma_{d^f_n}\text{-KMS}_{\operatorname{log}n}$-state.
    For $g=\sum_{i=1}^MT_{\mu_i}T_{\nu_i}^*+\sum_{k=1}^NT_{v_k}e_nT_{w_k}^*\in\Gamma_n$,
    we have $\psi (1_{Z(\mu_1)}\lambda_g)=1_{\mathfrak{S}_{E^f_n}}(g)\psi (1_{Z(\mu_1)}\lambda_g)$.
\end{lem}
\begin{proof}
    The KMS condition yields
    \[\psi ((1_{Z(\mu_1)}\lambda_g) 1)=n^{|\nu_1|-|\mu_1|}\psi (1 (1_{Z(\mu_1)}\lambda_g)),\]
    \[\psi ((1_{Z(\mu_1)}\lambda_g)1_{Z(\nu_1)})=n^{|\nu_1|-|\mu_1|}\psi (1_{Z(\nu_1)\cap Z(\mu_1)}\lambda_g).\]
    Thus,
    one has $\psi (1_{Z(\mu_1)}\lambda_g)=\delta_{\mu_1, \nu_1}\psi (1_{Z(\mu_1)}\lambda_g)$.
    If $\mu_1=\nu_1$, Lem. \ref{umu} shows $\psi (1_{Z(\mu_1)}\lambda_g)=1_{\mathfrak{S}_{E^f_n}(g)}\psi (1_{Z(\mu_1)}\lambda_g)$.
    Since $\mu_1\not=\nu_1$ implies $g\not\in\mathfrak{S}_{E^f_n}$, we complete the proof.
\end{proof}
\begin{lem}\label{gj}
Fix an arbitrary element $h\in\mathfrak{S}_{E^f_n}$.
    For any $g\in\Gamma_n$,
    there exist $c_h(g)\in \Gamma_n, h(g)\in\mathfrak{S}_{E^f_n}$ satisfying
    \[g=c_h(g)h(g),\quad c_h(g)h=hc_h(g).\]
\end{lem}
\begin{proof}
Since $h\in\mathfrak{S}_{E^f_n}$,
there is a finite set $F\subset E^f_n$ with $h\in\mathfrak{S}_F\subset\mathfrak{S}_{E^f_n}$.
%Recall  $\operatorname{supp}(h)=\{\mu\in E^f_n\;|\; h(\mu)\not=\mu\}$.
%Since the order of $h$ is finite,
%the set \[\operatorname{Orb}h:=\{h^n(\mu)\: |\; n\geq 0,\;\;\mu\in\operatorname{supp}(h)\}\subset E^f_n\] is finite,
    There is $N\in\mathbb{N}$ satisfying \[F \subset \{\mu\in E^f_n\;|\; |\mu|<N\}=(\bigsqcup_{|\nu|=N}Z(\nu))^c.\]
    For any $g\in\Gamma_n$,
    one has a presentation
    \[g=\sum_{i=1}^MT_{\mu_i}T_{\nu_i}^*+\sum_{k=1}^NT_{v_k}e_n T_{w_k}^*.\]
    Applying the following subdivisions
    \begin{align*}T_{\mu_i}T_{\nu_i}^*=&T_{\mu_i}(\sum_{i=1}^nT_iT_i^*+e_n)T_{\nu_i}^*\\
    =&T_{\mu_i}(\sum_iT_i(\sum_{j=1}^nT_jT_j^*+e_n)T_i^*+e_n)T_{\nu_i}^*\\
    =&T_{\mu_i}(\sum_{i}T_i(\sum_jT_j(\cdots(\sum_{l=1}^nT_lT_l^*+e_n)\cdots)T_j^*+e_n)T_i^*+e_n)T_{\nu_i}^*,
    \end{align*}
    we may assume that
    \[|\mu_i|, |\nu_i|\geq N,\quad \text{for}\; i=1,\cdots, M.\]
    This implies
    \[F\subset (\bigsqcup_{i=1}^MZ(\mu_i))^c=\{v_k\}_{k=1}^N,\quad F \subset (\bigsqcup_{i=1}^MZ(\nu_i))^c=\{w_k\}_{k=1}^N.\]
    Thus, one can take a partition
    \[\{v_k\}_{k=1}^N=F\sqcup \{v_s\}_{s=1}^L,\quad \{w_k\}_{k=1}^N=F\sqcup \{w_s\}_{s=1}^L.\]
    Define
    \[c_h(g):=\sum_{i=1}^MT_{\mu_i}T_{\nu_i}^*+\sum_{v\in F }T_ve_nT_v^*+\sum_{s=1}^LT_{v_s}e_nT_{w_s}^*\in \Gamma_n.\]
    It is obvious to see
    \[\pi(g)=\sum_{i=1}^MS_{\mu_i}S_{\nu_i}^*=\pi (c_h(g)), \quad h(g):=c_h(g)^{-1}g\in\mathfrak{S}_{E^f_n}.\]
    Note that $h( F)\subset F$, $c_h(g)|_F={\rm id}$ (i.e., $c_h(g)(F^c)=F^c$ ).
    Thus, one has
    \[c_h(g) h(v)=c_h(g)(h(v))=h(v),\quad hc_h(g)(v)=h(c_h(g)(v))=h(v),\;\text{for}\;v\in F,\]
    \[c_h(g) h(\mu)=c_h(g)(h(\mu))=c_h(g)(\mu),\quad hc_h(g)(\mu)=h(c_h(g)(\mu))=c_h(g)(\mu),\;\text{for}\;\mu\not\in F\]
    (i.e., $c_h(g)h=hc_h(g)$).
\end{proof}
\begin{cor}\label{uhoo}
    Every tracial state $\tau \colon C^*_r(\mathfrak{S}_{E^f_n})\to\mathbb{C}$ is $\Gamma_n$-invariant.
\end{cor}
\begin{proof}
    For $h\in\mathfrak{S}_{E^f_n}, g\in\Gamma_n$,
    Lem. \ref{gj} shows
    \[\tau(\lambda_{g^{-1}hg})=\tau(\lambda_{h(g)^{-1}c_h(g)^{-1}hc_h(g)h(g)})=\tau (\lambda_{h(g)^{-1}hh(g)})=\tau(\lambda_h).\]
\end{proof}

\begin{proof}[{Proof of Thm. \ref{kms}}]
First, we show (1).
For $\beta>\log n$,
a quasi-invariant measure with the Radon--Nikodym cocycle $e^{-\beta d^f_n}$ is given by
\[\sum_{\mu \in E^f_n}(1-ne^{-\beta})e^{-\beta |\mu|}\delta_\mu,\]
where $\delta_\mu : C(\partial E_n)\ni f\mapsto f(\mu)\in\mathbb{C}$ is the dirac measure.
If a quasi-invariant measure $m$ with the Radon--Nikodym cocycle $e^{-\beta d^f_n}$ satisfies $m(\{\mu\})\not=0$ for some $\mu\in E^f_n$, one has 
\[m(\{\mu\})=m(\{v_0\})e^{-\beta |\mu|}.\]
For the measurable sets $Z(\{i\})\backslash E^f_n$, ($i=1, \cdots, n$),
one has
\[m(Z(\{i\})\backslash E^f_n)=e^{-\beta}m(\partial E_n\backslash E^f_n),\quad \partial E_n\backslash E^f_n=\bigsqcup_{i=1}^nZ(\{i\})\backslash E^f_n,\]
and the assumption $\beta>\log n$ implies $m(\partial E_n\backslash E^f_n)=0$.
Thus, we conclude that 
\[\sum_{\mu \in E^f_n}(1-ne^{-\beta})e^{-\beta |\mu|}\delta_\mu\]
is the unique quasi-invariant measure with the Radon--Nykodim cocycle $e^{-\beta d^f_n}$ for $\beta>\log n$.

Let $\psi : C(\partial E_n)\rtimes_r \Gamma_n\to\mathbb{C}$ be a $\gamma_{d^f_n}-KMS_\beta$ state ($\beta>\log n$), and let $\tilde{\psi} : C(\partial E_n)\rtimes\Gamma_n \to C(\partial E_n)\rtimes_r \Gamma_n\xrightarrow{\psi}\mathbb{C}$ be its pull back.
By \cite[Cor. 1.4.]{Nes},
one has 
\[\tilde{\psi}(f\lambda_g) =\sum_{\mu \in E^f_n}(1-ne^{-\beta})e^{-\beta |\mu|}\delta_\mu f(\mu)\tilde{\tau}(\lambda_{{g_\mu}^{-1}gg_\mu}), \;\; f\in C(\partial E_n), \; g_\mu\in\Gamma_n, \;\; g_\mu (v_0)=\mu\in \partial E_n\]
for a tracial state $\tilde{\tau}$ on  $C^*((\partial E_n\rtimes \Gamma_n)_{v_0}^{v_0})=C^*((\Gamma_n)_{v_0})$.
Since $\tilde{\psi}$ comes from $\psi$,  \cite[Prop. 3.1.]{CNes} shows that $\tilde{\tau}$ factor through $C^*_r((\Gamma_n)_{v_0})$ (i.e., there exists  a tracial state $\tau$ satisfying $\tilde{\tau} : C^*((\Gamma_n)_{v_0})\to C_r^*((\Gamma_n)_{v_0})\xrightarrow{\tau} \mathbb{C}$.
This proves (1).

Next, we will show (2).
The same computation as above (i.e., $m(\{\mu\})=m(\{v_0\})e^{-\beta|\mu|}$) shows that there is the unique quasi-invariant measure with the Radon--Nykodim cocycle $e^{-\beta d^f_n}=n^{-d^f_n}$ given by 
\[m : C(\partial E_n)\to C(\partial E_n\backslash E_n^f)=C(\partial O_n)\subset M_{n^\infty}\xrightarrow{\text{trace}}\mathbb{C}.\]
Let $\tau : C^*_r(\mathfrak{S}_{E_n^f})\to \mathbb{C}$ be a tracial state.
By Cor. \ref{uhoo},
the composing $\tau$ and the conditional expectation $C^*_r(\Gamma_n)\to C^*_r(\mathfrak{S}_{E_n^f})$ gives a tracial state of $C^*_r(\Gamma_n)$ which we also denote by $\tau$.
Consider the constant measurable field $\{m, \{\tau |_{C^*_r((\partial E_n\rtimes\Gamma_n)_x^x)}\}_{x\in \partial E_n}\}$.
Since $\tau (g\cdot g^{-1})=\tau(\cdot)$ for $g\in\Gamma_n$,
\cite[Thm. 1.3.]{Nes} and \cite[Prop. 2.10, 3.1.]{CNes} shows that the state $\psi=\int_{\partial E_n}\tau dm$ is a $\gamma_{d^f_n}-KMS_{\log n}$ state.

Finally, Lem. \ref{ummu} shows that every $\gamma_{d^f_n}-KMS_{\log n}$-KMS state $\psi$ must satisfy $\psi (1_{Z^\infty(\mu)\lambda_g})=m(Z^\infty(\mu))\tau(1_{\mathfrak{S}_{E^f_n}}(g)\lambda_g)$, and this completes the proof.

\end{proof}

\end{document}